\documentclass[twoside]{article}
\usepackage{amsmath,amssymb,amsfonts,amsthm,amscd}
\usepackage[latin9]{inputenc}
\usepackage[letterpaper]{geometry}
\usepackage{enumerate}
\usepackage{esint}

\usepackage{fancyhdr}
\makeatletter
  \theoremstyle{plain}

\usepackage{amsfonts}\usepackage{amsthm}\usepackage{amscd}
\usepackage{graphics}
\usepackage{color}
\usepackage{mathrsfs}
\usepackage{indentfirst}
\usepackage{url}         
\usepackage{colonequals}
\usepackage{authblk} 
\usepackage{a4wide}
\newtheorem{theorem}{Theorem}[section]
\newtheorem{proposition}[theorem]{Proposition}

\newtheorem{lemma}[theorem]{Lemma}
\theoremstyle{definition}
\newtheorem{remark}[theorem]{Remark}
\newtheorem{scheme}[theorem]{Scheme}

\newtheorem{assumption}[theorem]{Assumption}

\def\div{\mathop{\mathrm{div}}\nolimits}
\def\!{\mathop{\mathrm{!}}}
\def\argmin{\mathop{\mathrm{argmin}}}

\providecommand{\abs}[1]{\lvert#1\rvert}

\makeatother

  \providecommand{\lemmaname}{Lemma}

\def\PSopt{P_{\mathrm{opt}}}

\def\R{\mathbb{ R}}

\def\C{\mathcal{ C}}

\def\Cc{\mathcal{ C}}

\def\P{\mathcal{P}}

\def\F{\mathcal{F}}

\def\K{\mathcal{K}}

\def\K{\mathcal{K}}

\def\T{\mathcal{T}}
\def\W{\mathcal{W}}

\def\b{\textbf{b}}
\def\x{\textbf{x}}
\def\y{\textbf{y}}

\def\z{\mathbf{z}}

\def\w{\mathbf{w}}

\newlength{\boxwidth}
\title{On the fundamental solution and a variational formulation for a degenerate diffusion of Kolmogorov type}
\author[1]{Manh Hong Duong\thanks{Corresponding author: M. H. Duong}}
\author[2] {Hoang Minh Tran}
\affil[1]{Department of Mathematics, Imperial College London, London SW7 2AZ, 
UK. Email: m.duong@imperial.ac.uk.}
\affil[2]{Data Analytics Department, Esmart Systems, H\aa kon Melbergs vei 16, 1783 Halden, Norway. Email: hoangtm.fami@gmail.com}
\begin{document}
\maketitle 
\begin{abstract}
In this paper, we construct the fundamental solution to a degenerate diffusion of Kolmogorov type and develop a time-discrete variational scheme for its adjoint equation. The so-called mean squared derivative cost function plays a crucial role occurring in both the fundamental solution and the variational scheme. The latter is implemented by minimizing a free energy functional with respect to the Kantorovich optimal transport cost functional associated with the mean squared derivative cost. We establish the convergence of the scheme to the solution of the adjoint equation generalizing previously known results for the Fokker-Planck equation and the Kramers equation.
\let\thefootnote\relax\footnotetext{2010 Mathematics Subject Classification. Primary: 49S05; Secondary: 49J40, 35Q84. \\ Key words and phrases. Generalized gradient flow structure, variational methods, hypo-elliptic PDEs, optimal transport.}
\end{abstract}
\section{Introduction}
In this paper, we are interested in the following partial differential equation
\begin{equation}
\label{eq: gKramers}
\partial_t\rho(t,x_1,\ldots,x_n)=-\sum_{i=2}^{n}x_{i}\cdot\nabla_{x_{i-1}}\rho+\div_{x_n}(\nabla V(x_n)\rho)+\Delta_{x_n}\rho
\end{equation}
and its adjoint equation when $V\equiv 0$
\begin{equation}
\label{eq3}
\partial_{t}f(t,x_{1},\ldots,x_{n})=\sum_{i=2}^{n}x_{i}\cdot\nabla_{x_{i-1}}f+\Delta_{x_n}f.
\end{equation}
In the above equations, the unknowns are $\rho=\rho(t,x_1,\ldots, x_n)$ and $f=f(t,x_1,\ldots,x_n)$ where $t>0$ and $\x=(x_1,\ldots,x_n)\in \R^{dn}$ are time and spatial co-ordinates, respectively. The notations $\nabla, \div$ and $\Delta$ denote the gradient, the divergence and the Laplacian operators respectively. The subscripts in these operators indicate that they only act on the corresponding variables. The two equations are complemented with initial data.  The initial data for \eqref{eq: gKramers} is a probability measure, $\rho_0$, on $\R^{dn}$. Since the right-hand side of \eqref{eq: gKramers} is a divergence form, noting that the summation term can be written as $-\sum_{i=2}^{n} \div_{x_{i-1}}(x_i \rho)$, for each $t>0$, $\rho(t)$ is also a probability measure on $\R^{dn}$. Since we will be interested in the fundamental solution of \eqref{eq3}, the initial data for \eqref{eq3} is a Dirac measure $\delta_{\y}$ for $\y\in\R^{dn}$. Two special cases of \eqref{eq: gKramers}, that correspond to $n=1$  and $n=2$ respectively, are the Fokker-Planck equation
\begin{equation}
\label{eq: FP eqn}
\partial_t\rho=\div(\nabla V\rho)+\Delta\rho,
\end{equation}
and the Kramers equation
\begin{equation}
\label{eq: Kramers eqn}
\partial_t\rho=-x_2\cdot\nabla_{x_1}\rho+\div_{x_2}(\nabla V(x_2)\rho)+\Delta_{x_2}\rho.
\end{equation}
Their corresponding adjoint equations with $V\equiv 0$ are, respectively, the diffusion equation $\partial_t f=\Delta f$ and the ultra-parabolic equation 
\begin{equation}
\label{eq: adjKramers}
\partial_t f=x_2\cdot\nabla_{x_1} f+\Delta_{x_2}f.
\end{equation}
Equations \eqref{eq: gKramers} and \eqref{eq3}, particularly the Fokker-Plank equation and the Kramers equation, play an important role in  statistical mechanics \cite{Ris84}, reaction-rate theory \cite{HTB90} and mathematical finance \cite{LPP02,Pascucci2005}. For instance, in the context of statistical mechanics, \eqref{eq: gKramers} has been used to describe the time evolution of the probability density function of a many-particle system. Equation \eqref{eq: gKramers} with $n>2$ can also be viewed as a simplified model of 
a finite Markovian approximation for the generalised Langevin dynamics \cite{OP11, Duong15NA} or a model of a harmonic chains of oscillators that arises in the context of non-equilibrium statistical mechanics \cite{EH00,BL08, DelarueMenozzi10}. 

Equations \eqref{eq: gKramers} and \eqref{eq3} belong to a class of degenerate diffusions of Kolmogorov type where the Laplacian acts only in some of variables. Two main issues that have been getting a lot of research interest in this class is (i) constructing the fundamental solution of \eqref{eq3} and (ii) developing a time-discrete variational scheme for \eqref{eq: gKramers}. To motivate our work, let us discuss relevant literature on these issues for the Fokker-Planck equation and the Kramers equation. Regarding the first issue, it is a classical result that the diffusion equation $\partial_t f=\Delta f$ has the fundamental solution
\begin{equation}
\label{eq: fdFP}
\Phi_{\rm FP}(t,\x,\y)=\frac{1}{(4\pi t)^\frac{d}{2}}\exp\Big(-\frac{C^{\rm FP}(\x,\y)}{4t}\Big)\quad\text{with}\quad C^{\rm FP}(\x,\y)=|\x-\y|^2.
\end{equation}
In the seminal paper \cite{Kol34}, Kolmogorov show that
\begin{align}
\label{eq: fdKR}
&\Phi_2(t,x_1,y_1;x_2,y_2)=\Bigg(\frac{\sqrt{3}}{2\pi t^2}\Bigg)^{d}\exp\left(-\frac{C^{\rm KR}_t(\x,\y)}{4t}\right)~\text{with}\notag
\\& C^{\rm KR}_t(\x,\y)=|y_2-y_1|^2+12\Big|\frac{x_2-x_1}{t}-\frac{y_1+y_2}{2}\Big|^2
\end{align}
is the fundamental solution to the ultra-parabolic equation \eqref{eq: adjKramers}. It is this equation that was H\"{o}rmander's starting point to develop the hypo-elliptic theory \cite{Hormander67}, which has become a powerful tool in the theory of partial differential equations, see e.g. \cite{Bra14} for discussions. Since Kolmogorov's paper there has been a considerable amount of work on extending his result to other hypoelliptic equations including \eqref{eq3}, see e.g., \cite{Weber51, K72, Polidoro94, FP05, DelarueMenozzi10, BM15}.  We refer to the mentioned papers and references therein for more  information on this direction.

Regarding the second issue, the functions $C^{\rm FP}$ and $C^{\rm KR}$ also play a crucial role in the more recent development of time-discrete variational schemes for the Fokker-Planck equation and the Kramers equation, respectively. In the seminal paper \cite{JKO98}, Jordan-Kinderlehrer-Otto prove a remarkable result that the Fokker-Planck equation \eqref{eq: FP eqn} can be seen as a gradient flow of the free energy, which is the sum of the potential energy $\int V \rho$ and the Boltzmann entropy $\int \rho\log\rho$, with respect to the Wasserstein distance on the space of probability measures with finite second moments. Let $\mu,\nu\in\P_2(X)$ be two probability measures on some Euclidean space $X$ having finite second moments. Throughout this paper, $\Gamma(\mu,\nu)$ denotes the set of all probability measures on $X \times X$ having $\mu$ and $\nu$ as first and second marginals. The Wasserstein distance $W_2(\mu,\nu)$ between them is defined via
\begin{equation}
\label{eq: W2 distance}
W^2_2(\mu,\nu):=\inf_{\gamma\in\Gamma(\mu,\nu)}\int_{\R^d\times\R^d}|\x-\y|^2\,\gamma(d\x,d\y).
\end{equation}
The main ingredient in \cite{JKO98} is the following variational approximation scheme, which is now known as the JKO-scheme,
\\ \ \\
\textbf{JKO-scheme:} Let $h>0$ be a time-step. Define $\rho^0_h:=\rho_0$. Then, for each $k=1,2,...$,  $\rho^k_h$ is determined as 
\begin{equation}
\label{eq: JKO-scheme}
\rho^{k}_h=\argmin_{\rho\in\P_2(\R^d)}\left\{\frac{1}{2h}W^2(\rho^{k-1}_h,\rho)+\int_{\R^d}\big(V+\log\rho)\rho\,d\x\right\}.
\end{equation} 

The main result in \cite{JKO98} then states that, after an appropriate interpolation, the sequence $\{\rho^k_h\}$ constructed from the JKO-scheme converges to the solution of the Fokker-Planck equation \eqref{eq: FP eqn}. This result has sparked a large body of research in the last two decades in the field of partial differential equations linking the field to some other branches of mathematics such as optimal transport theory, geometric measure theory and functional inequalities, see monographs \cite{Vil03, AGS08,Vil09} for great expositions on the development. 

Inspired by the JKO-scheme, Huang \cite{Hua00} and then Duong et al. \cite{DPZ13a} have established different approximation schemes for the Kramers equation \eqref{eq: Kramers eqn}. The challenge here is that the techniques in \cite{JKO98} can not be directly applied as the Kramers equation is neither a gradient flow nor a Hamiltonian flow even though these schemes are of the same form as in \eqref{eq: JKO-scheme}. The free energy functional is the same, but instead of the Wasserstein distance, the following Monge-Kantorovich optimal transport cost has been used in these papers
\begin{equation}
\label{eq: MKcost}
\tilde{W}^2_h(\mu,\nu)=\inf_{\gamma\in\Gamma(\mu,\nu)}\int_{\R^{2d}\times\R^{2d}}C^{\rm KR}_h(\x,\y)\,\gamma(d\x,d\y),
\end{equation}
where $C^{\rm KR}_h(\x,\y)$ is defined in \eqref{eq: fdKR}.
The cost functional $\tilde{W}_h$ has also been used to construct time-discrete variational schemes for other evolution equations such as the system of isentropic Euler equations \cite{GW09, Westdickenberg10} and the compressible Euler equations \cite{CSW2014TMP}. 

Through the above discussions, we notice a similarity between the Fokker-Planck equation and the Kramers equation. That is the role of the cost functions $C^{\rm FP}$ and $C_t^{\rm KR}$. They appear both in the discrete-time variational scheme and the fundamental solution. In addition, these cost functions also satisfy the following property: they minimize the velocity and acceleration integrals respectively
\begin{align*}
&C^{\rm FP}(\x,\y)=h\min_{\xi}\Big\{\int_0^h |\dot{\xi}(t)|^2\,dt: \xi\in C^1([0,h],\R^d)~\text{such that}~\xi(0)=\x,~\xi(h)=\y\Big\}\quad \text{and}
\\& C^{\rm KR}_h(\x,\y)=h\min_{\xi}\Big\{\int_0^h |\ddot{\xi}(t)|^2\,dt:\xi\in C^2([0,h],\R^d)~\text{such that}~(\xi,\dot{\xi})(0)=\x=(x_1,x_2),
\\&\hspace*{8cm}~(\xi,\dot{\xi})(h)=\y=(y_1,y_2)\Big\}.
\end{align*}

Recently in \cite{DuongTran16} we have studied the minimization problem 
\begin{equation}
\label{eq: Ch cost}
\Cc_t(\x, \y):=t \inf\limits_{\xi}\int_0^t|{\xi}^{(n)}(s)|^{2}\,ds,
\end{equation}
where $\mathbf{x}=(x_{1},\ldots,x_{n})\in\R^{dn},\mathbf{y}=(y_1,\ldots,y_n)\in\R^{dn}$ and the infimum is taken over all curves $\xi\in C^{n}([0,T],\R^d)$ that satisfy the boundary conditions
\begin{equation}
\label{eq: boundary conditions}
(\xi,\dot{\xi},\ldots,\xi^{(n-1)})(0)=(x_1,x_{2},\ldots,x_n)\quad\text{and}\quad (\xi,\dot{\xi},\ldots,\xi^{(n-1)})(t)=(y_1,y_{2},\ldots,y_{n}).
\end{equation}

The optimal value $\Cc_t(\x,\y)$ is called the mean squared derivative cost function and has been found to be useful in the modelling and design of various real-world systems such as motor control, biometrics and online-signatures and robotics, see \cite{DuongTran16} for further discussion.

Inspired by the role of $C^{\rm FP}$ and $C_t^{\rm KR}$ for the Fokker-Planck equation and the Kramers equation as discussed above, it is natural to ask
\begin{enumerate}[(\textbf{Q}1)]
\item Is the function
\begin{equation}
\label{eq: formula Phi}
\Phi(t,\x,\y):=\frac{\beta_d}{t^\frac{n^2 d}{2}}\exp\Big(-\frac{\Cc_t(\x,\y)}{4t}	\Big),
\end{equation}
where $\beta_d$ is the normalising constant, the fundamental solution to Equation \eqref{eq3}?
\item Can Equation \eqref{eq: gKramers} be approximated by a discrete-time variational scheme in the spirit of the JKO-scheme where $C^{\rm FP}(\x,\y)$ is replaced by $\Cc_h(\x,\y)$?
\end{enumerate}
The aim of the present paper is to provide affirmative answers to these questions. We now describe our main results.

\subsection{Main results of the present paper}
Our first main result is the following theorem about the fundamental solution to \eqref{eq3}. 
\begin{theorem}
\label{theo: maintheorem1}
The function \emph{$\Phi(t,\x,\y)$} defined in \eqref{eq: formula Phi} is the fundamental solution to \eqref{eq3}. That is, for each \emph{$\y$, $\Phi(t,\x,\y)$}, as a function of $t$ and \emph{$\x$}, satisfies \eqref{eq3}. In addition,
\emph{
\begin{equation}
\label{eq: initial diract}
\lim\limits_{t\to 0}\Phi(t,\x,\y)=\delta_{\x=\y}.
\end{equation}}
\end{theorem}
This theorem together with the explicit formula for the mean squared derivative cost function $\Cc_t(\x,\y)$ in \cite[Theorem 1.2]{DuongTran16} (see also Theorem 2.1 below) provide an explicit formula for the fundamental solution $\Phi(t,\x,\y)$.

Our second main result is concerned with a variational formulation of \eqref{eq: gKramers}. We first introduce the approximation scheme in the spirit of the JKO-scheme. Let $h>0$ be given and $\Cc_h(\x,\y)$ be the mean square derivative cost function defined in \eqref{eq: Ch cost}. Let $\mu$ and $\nu$ be two probability measures on $\R^{dn}$ having finite second moments. The Monge-Kantorovich optimal transport cost $\W_h(\mu,\nu)$ between $\mu$ and $\nu$ is defined by
\begin{equation}
\label{eq: MK cost}
\W_h(\mu,\nu)=\inf_{\gamma\in\Gamma(\mu,\nu)}\int_{\R^{dn}\times \R^{dn}}\Cc_h(\x,\y)\,\gamma(d\x d\y).
\end{equation}

The variational approximation scheme of this paper is constructed as follows.
\begin{scheme}
\label{scheme}
Let $\rho_0^h:=\rho_0$. For $k\geq 1$, define $\rho_k^h$ as the solution of the minimization problem
\begin{equation}
\label{eq: min prob}
\min_{\rho\in \P_2(\R^{dn})} \frac{1}{2h}\W_h(\rho_{k-1}^h,\rho)+\int_{\R^{dn}} \big(V(x_n)+\log\rho\big)\rho\,d\x.
\end{equation}
\end{scheme}
Next we introduce the concept of a weak solution of \eqref{eq: gKramers}.  A function $\rho\in L^1(\R^+\times\R^{dn})$ is called a weak solution of equation~\eqref{eq: gKramers} with initial datum $\rho_0\in \P_2(\R^{dn})$ if it satisfies the  following weak formulation of~\eqref{eq: gKramers}:
\begin{multline}
 \label{weakKReqn}
\int_0^\infty\int_{\R^{dn}}\Big[\partial_t \varphi +\sum_{i=2}^n x_i\cdot \nabla_{x_{i-1}} \varphi-\nabla_{x_n}V(x_n)\cdot\nabla_{x_n} \varphi +\Delta_{x_n} \varphi\Big]\rho(t,\x)\,d\x \,dt\\=-\int_{\R^{dn}}\varphi(0,\x)\rho_0(\x)\,d\x~~\text{for all}~~\varphi\in C_c^\infty(\R\times \R^{dn})
\end{multline}
Throughout the paper we make the following assumptions.
\begin{assumption}
\label{assumpt}
\begin{equation}
V\in C^2(\R^d),\  V(x)\geq 0 \text{ for all }x\in \R^d,\label{assumpt1}
\end{equation}
and there exists a constant $C>0$ such that for all $x_1,x_2\in \R^d$
\begin{equation}
\label{assumpt2}
\abs{\nabla V(x_1)-\nabla V(x_2)}\leq C\abs{x_1-x_2}.
\end{equation}
\end{assumption}

The second main result of the paper is about the convergence of the approximation scheme to a weak solution of \eqref{eq: gKramers}.
\begin{theorem}
\label{theo:maintheorem 2}
Suppose that $V$ satisfies Assumption \ref{assumpt}. Let $\rho_0\in \P_2(\R^{dn})$ satisfy \emph{$\int_{\R^{dn}}(V(x_n)+\log\rho_0)\,\rho_0\,d\x < \infty$}. For any $h>0$ sufficiently small, let $\rho_k^h$ be the sequence of the solutions of the \eqref{eq: min prob}. For any $t\geq0$, define the piecewise-constant time interpolation
\begin{equation}
\rho^h(t)=\rho_k^h \qquad \text{ for } (k-1)h<t\leq kh.\label{interpolation}
\end{equation}
Then for any $T>0$,
\begin{equation}
\rho^h\rightharpoonup \rho ~\text{ weakly in }~ L^1((0,T)\times\R^{dn})~\text{ as } h\to0,\label{weaklyconverence}
\end{equation}
where $\rho$ is the unique weak solution of Equation \eqref{eq: gKramers} with initial value $\rho_0$. Moreover
\begin{equation}
\rho^h(t)\rightarrow\rho(t) ~ \text{ weakly in} ~ L^1(\R^{dn}) ~ \text{ as }~ h\to0 ~ \text{ for any }~ t>0, \label{pointwiseconvergence}
\end{equation}
and as $t\to0$,
\begin{equation}
\rho(t)\rightarrow\rho_0 ~\text{ in } ~ L^1(\R^{dn}). \label{intinialconvergence}
\end{equation}
\end{theorem}
\subsection{Comparison to related work}
Our work is twofold: to construct the fundamental solution to equation \eqref{eq3} (Theorem \ref{theo: maintheorem1}) and to develop a time-discrete variational scheme for equation \eqref{eq: gKramers} which is the adjoint equation of \eqref{eq3} with an additional external force field (Theorem \ref{theo:maintheorem 2}). The mean squared derivative cost function \eqref{eq: Ch cost} plays a central role appearing both in the fundamental solution and in the approximation scheme. We now give further comments on these issues.

\textit{On the fundamental solution of \eqref{eq3}.} The fundamental solution to \eqref{eq3} is not new. As mentioned above, it has been shown in \cite{Weber51, Hormander67, K72, Polidoro94, FP05, DelarueMenozzi10, BM15}. In these papers, the fundamental solution is constructed using various methods such as parametrix and Fourier-transform. In this paper, we provide a direct verification based on elementary combinatorial and linear algebra techniques. We use explicit formulas for the mean squared derivative cost function that we obtained in our recent work \cite{DuongTran16}.  Our method is closed to \cite{K72}. However, this work does not represent the fundamental solution of \eqref{eq3} in terms of the mean squared derivative cost function. The reference~\cite{DelarueMenozzi10} provides an implicit representation via the controllability property of a differential system but this work does not address the variational formulation of \eqref{eq: gKramers}.

\textit{Variational formulation for equation \eqref{eq: gKramers}.} Theorem \ref{theo:maintheorem 2} is a generalisation of the main results of \cite{JKO98, Hua00,DPZ13a} for an arbitrary $n$. Even though the standard procedure in these papers can be used to prove the theorem, two additional difficulties arise. The first difficulty is how to select an appropriate perturbation flow to derive the Euler-Lagrange equation for the sequence $\rho_k^h$. The other difficulty is to prove the convergence of the scheme to a solution of \eqref{eq: gKramers} which amounts to show the error terms vanish as $h\to 0$. These difficulties can be solved by using the explicit formulas involving the derivative mean squared cost function derived in our previous work \cite{DuongTran16}.  We also note that~\cite{Hua00} and~\cite{DPZ13a} considered the full Kramers equation (i.e., with $n=2$) with an external force field $-\nabla_{x_1}U(x_1)$ where $U=U(x_1)$ is a given potential. In this case, the right-hand side of~\eqref{eq: gKramers} has an additional term $\div_{x_2}(\nabla_{x_1} U \rho)$. \cite{Hua00} has dealt with this term by adding $\frac{1}{h}\int U(x_1)\rho\,d\x$ into the free-energy functional resulting in an unusual scale. In contrast, this term has been encoded in the cost function in~\cite{DPZ13a}. In the present work for arbitrary number of variables $n$, we have made an assumption that $V$ depends only on the last co-ordinate $x_n$. Due to this assumption Equation~\eqref{eq: gKramers} resembles the Fokker-Planck equation in the last co-ordinate and the cost function $\C_t$ depends only on the $n$-th order derivative of $\xi$ giving rise to a controllable formula. It is not clear to us at the moment how to adapt~\cite{Hua00,DPZ13a} to deal with a more general case where $V$ depends on more than the last co-ordinate or when the co-ordinates are coupled in a more complex way as in~\cite{EH00,DelarueMenozzi10, OP11}. We leave this issue for future research.

\textit{Microscopic interpretation of the variational scheme.} The main results in recent papers \cite{ADPZ11, DLR13, DPZ13b, erbar2015} show that Scheme \ref{scheme} with $n=1$ and $n=2$ for the Fokker-Planck equation and the Kramers equation, respectively,  can be derived from large-deviation principles of empirical measures associated to underlying stochastic processes. These results provide, among other things, microscopic interpretations for Scheme \ref{scheme} in these cases. We expect that these results can be extended to the general case $n$, but we will not follow this direction in this paper.
\subsection{Organization of the paper}
The rest of the paper is organised as follows. In Section \ref{sec: properties of C}, we summarize the main properties of the mean squared cost function in \cite{DuongTran16}. In Section \ref{sec: fundamental} and Section \ref{sec: variatonal} we prove Theorem \ref{theo: maintheorem1} and Theorem \ref{theo:maintheorem 2}, respectively. Finally, Appendix \ref{sec: App} contains proofs of technical lemmas.
\section{The mean square derivative cost function}
\label{sec: properties of C}
In this section, we collect relevant results on the mean squared derivative cost function in \cite{DuongTran16}.
\begin{theorem}\cite[Theorem 1.2]{DuongTran16}
\label{theo: C}
The mean square derivative cost function \emph{$\Cc_t(\x,\y)$} is given by
\emph{
\begin{align*}
\Cc_t(\x,\y)=\,t^{2-2n}\,[\b(t,\x,\y)]^{T}M\b(t,\x,\y)
\end{align*}}
where 
\begin{equation}
\label{eq: b}
\emph{\b(t,\x,\y)}=\begin{pmatrix}y_{1}-x_{1}-\frac{t}{1}x_{2}-...-\frac{t^{n-1}}{(n-1)!}x_{n}\\
\vdots\\
t^{i-1}\Big(y_{i}-\sum_{j=i}^{n}\frac{t^{j-i}}{(j-i)!}x_{j}\Big)\\
\vdots\\
t^{n-1}(y_{n}-x_{n})
\end{pmatrix}
\end{equation}
and $M=BA^{-1}$ with
\begin{align}
\label{A}
&A=\left[\begin{array}{ccc}
1 & ... & 1\\
\begin{pmatrix}n\\
1
\end{pmatrix} & ... & \begin{pmatrix}2n-1\\
1
\end{pmatrix}\\
\vdots & \vdots & \vdots\\
k!\begin{pmatrix}n\\
k
\end{pmatrix}& ... & k!\begin{pmatrix}2n-1\\
k
\end{pmatrix}\\
\vdots & \vdots & \vdots\\
(n-1)!\begin{pmatrix}n\\
n-1
\end{pmatrix} & ... & (n-1)!\begin{pmatrix}2n-1\\
n-1
\end{pmatrix}
\end{array}\right]\quad\text{and}
\\\nonumber
\\& B_{ki}=\begin{cases}
(-1)^{n-k}\frac{(n+i-1)!}{(k+i-n-1)!}, & \quad \text{if} \quad k+i\ge n+1\\
0 & \quad \text{if} \quad  k+i<n+1.
\end{cases}\label{B}
\end{align}
\end{theorem}
Moreover, the matrix $A$ and its inverse have explicit $LU$-decompositions given in the following theorem.
\begin{theorem}\cite[Theorem 1.3]{DuongTran16}
\label{theo: LU of A 2}
\begin{enumerate}[(1)]
\item  $A=LU$, where $U$ and $L$ are defined as follows
\begin{equation}
\label{U}
U_{ij}= \begin{cases}
\frac{(j-1)!}{(j-i)!} & \text{ if }j\ge i,\\
0 & \text{otherwise},
\end{cases} \quad\text{and}\quad
L_{kj}=\begin{cases}
\begin{pmatrix}k-1\\
j-1
\end{pmatrix}\frac{n!}{(n-k+j)!} & \text{ if }j\le k,\\
0 & \text{otherwise}.
\end{cases}
\end{equation}
\item The inverse of $A$ is given by the product of the following two matrices:
\begin{align*}
&U^{-1}_{ij}=\begin{cases}
\frac{(-1)^{i+j}}{(i-1)!(j-i)!} & \text{ if }j\ge i,\\
0 & \text{otherwise}
\end{cases} \qquad \qquad\text{and}
\\
\\& L^{-1}_{ji}=\begin{cases}
(-1)^{j-i}\frac{(j-1)!}{(i-1)!}\begin{pmatrix}
n+j-i-1\\
j-i
\end{pmatrix}\nonumber& \text{ if }j\ge i\\
0 & \text{otherwise}.
\end{cases}
\end{align*}
\end{enumerate}
\end{theorem}
Note that throughout this paper, all the matrices $A, B, M, L$ and $U$ are of order $dn$. Each entry of these matrices should be understood as a $d$-dimensional matrix that is equal to the entry multiplies with the $d$-dimensional identity matrix $I_d$. For instance, $A_{ij}$, for $1\leq i,j\leq n$ should be understood as the matrix $A_{ij}I_d$. The multiplication of matrices are carried out as the multiplication of block matrices.

In the sequel sections, we also need the following the property of the cost function $\Cc_t(\x,\y)$.
\begin{lemma} There exists a constant $K>0$ independent of $t$ such that
\emph{
\begin{equation}
\label{ineqs:qpC}
|\y-\x|^2\leq K\Big[\Cc_t(\x,\y) + t^2 (|\x|^2+|\y|^2)\Big].
\end{equation}}
\end{lemma}
\begin{proof}
It follows from the formula for $\C_t(\x;\y)$ that only the symmetric part $M_s$ of $M$ contributes to $\C_t$. We have
\begin{equation*}
\Cc_t(\x,\y)=t^{2-2n} \b^T\, M_s \, \b= t^{2-2n} | M_s^{1/2}\b|^2=| \z |^2
\end{equation*}
where
\begin{equation*}
t^{n-1}\z=M_s^{1/2}\b=Q\b, \quad Q=M_s^{1/2}.
\end{equation*}
By the definition of $\b$, we have
\begin{equation*}
\b=\T (\y-\x-\w)\quad\text{where}\quad w_k=\sum_{j=k+1}^n\frac{t^{j-k}}{(j-k)!}x_j, \quad \T=\mathrm{diag}(1,\ldots,t^{n-1}).
\end{equation*}
Therefore 
\[
\y-\x-\w=t^{n-1}(\T)^{-1}Q^{-1}\z=\mathrm{diag}(t^{n-1},\ldots, 1)\, Q^{-1}\, \z=\bar{\T}\, Q^{-1}\, \z
\]
where $\bar{\T}=\mathrm{diag}(t^{n-1},\ldots, 1)$, which implies that $\y-\x=\bar{\T}Q^{-1}\,\z+\w$. 

By Cauchy-Schwarz inequality, for $t$ sufficiently small, we have $\|\bar{\T}\|\leq K$ and $\|\w\|^2\leq K t^2 \|\x\|^2$ for some constant $K>0$. We use the notation $K$ to denote a universal constant that may change from line to line. Therefore, we get
\begin{align*}
|\y-\x|^2&=|\bar{\T}\, Q^{-1}\,\z+\w|^2
\\&\leq 2\Big(|\bar{\T}\, Q^{-1}\,\z|^2+|\w|^2\Big)
\\&\leq 2\Big(\|\bar{\T}\|^2\|Q^{-1}\|^2 |\z|^2+|\w|^2\Big)
\\&\leq K\Big( \Cc_t(\x,\y) + t^2 |\x|^2\Big)
\\&\leq K\Big[\Cc_t(\x,\y) + t^2 (|\x|^2+|\y|^2)\Big].
\end{align*}
This finishes the proof of this lemma.
\end{proof}
\section{The fundamental solution of the adjoint equation \eqref{eq3}}
\label{sec: fundamental}
In this section we prove Theorem \ref{theo: maintheorem1}. The proof consists of two main steps which are Proposition \ref{prop: Phi and C} and Proposition \ref{prop: aux} below.
\begin{proposition}
\label{prop: Phi and C}
Let \emph{$\Phi(t,\x,\y)$} be defined as in \eqref{eq: formula Phi}. Then it is a solution of \eqref{eq3} if and only if \emph{$\Cc_t(\x,\y)$} satisfies the following equation
\begin{equation}
\label{eq: eqn for C}
\partial_{t}\Cc=\frac{\Cc}{t}+\sum_{i=2}^{n}x_{i}\cdot\nabla_{x_{i-1}}\Cc-\frac{1}{4t}|\nabla_{x_n}\Cc|^{2}+\Delta_{x_n}\Cc-2dn^2.
\end{equation}
\end{proposition}
\begin{proof}
Let $\alpha(n,d)=\frac{n^2d}{2}$. We compute each term in \eqref{eq3} from the representation of $\Phi$ in \eqref{eq: formula Phi}. For the simplicity of notation, in the following computations, we denote $\Cc:=\Cc_t(\x,\y)$.

First, we calculate the time-derivative of $\Phi$.
\begin{align}
\partial_{t}\Phi & =-\alpha(n,d)\frac{\beta_{d}}{t^{\alpha(n,d)+1}}\exp\left[-\frac{1}{4t}\Cc\right]+\frac{\beta_{d}}{t^{\alpha(n,d)}}\exp\left[-\frac{1}{4t}\Cc\right]\left(\frac{1}{4t^{2}}\Cc-\frac{1}{4t}\partial_{t}\Cc\right)\nonumber
\\& =\frac{\beta_{d}}{t^{\alpha(n,d)+1}}\exp\left[-\frac{1}{4t}\Cc\right]\left(-\alpha(n,d)+\frac{1}{4t}\Cc-\frac{1}{4}\partial_{t}\Cc\right).\label{eq: temp1}
\end{align}
Next, we calculate $\sum_{i=2}^n x_i\cdot\nabla_{x_{i-1}}\Phi$. We have
\begin{equation*}
x_{i}\cdot\nabla_{x_{i-1}}\Phi =-\frac{\beta_{d}}{4t^{\alpha(n,d)+1}}\exp\left[-\frac{1}{4t}\Cc\right] x_{i}\cdot\nabla_{x_{i-1}}\Cc,
\end{equation*}
from which by taking the sum over $i$ from $2$ to $n$, we get
\begin{equation}
\sum_{i=2}^n x_i\cdot\nabla_{x_{i-1}}\Phi=-\frac{\beta_{d}}{4t^{\alpha(n,d)+1}}\exp\left[-\frac{1}{4t}\Cc\right]\sum_{i=2}^n x_i\cdot \nabla_{x_{i-1}}\Cc.\label{eq: temp2}
\end{equation}
The Laplacian, $\Delta_{x_n}\Phi$, with respect to variable $x_n$ is computed analogously
\begin{equation}
\Delta_{x_n}\Phi=\div_{x_n}\Big(\nabla_{x_n}\Phi\Big)=-\frac{1}{4}\frac{\beta_{d}}{t^{\alpha(n,d)+1}}\exp\left[-\frac{1}{4t}\Cc\right]\left(-\frac{1}{4t}|\nabla_{x_{n}}\Cc|^{2}+\Delta_{x_n}\Cc\right).\label{eq: temp3}
\end{equation}
It follows from \eqref{eq: temp1}, \eqref{eq: temp2} and \eqref{eq: temp3} that $\Phi$ is a solution of \eqref{eq3} if and only if
\begin{align*}
&\frac{\beta_{d}}{t^{\alpha(n,d)+1}}\exp\left[-\frac{1}{4t}\Cc\right]\left(-\alpha(n,d)+\frac{1}{4t}\Cc-\frac{1}{4}\partial_{t}\Cc\right)
\\&\qquad=-\frac{\beta_{d}}{4t^{\alpha(n,d)+1}}\exp\left[-\frac{1}{4t}\Cc\right]\sum_{i=2}^n x_i\cdot \nabla_{x_{i-1}}\Cc\\
 &\qquad\qquad -\frac{\beta_{d}}{4t^{\alpha(n,d)+1}}\exp\left[-\frac{1}{4t}\Cc\right]\left(-\frac{1}{4t}|\nabla_{x_{n}}\Cc|^{2}+\Delta_{n}\Cc\right)\\
 &\qquad=\frac{\beta_{d}}{4t^{\alpha(n,d)+1}}\exp\left[-\frac{1}{4t}\Cc\right]\left(-\sum_{i=2}^{n}x_i\cdot \nabla_{x_{i-1}}\Cc+\frac{1}{4t}|\nabla_{x_{n}}\Cc|^{2}-\Delta_{n}\Cc\right).
 \end{align*}
The above equality is equivalent to
\begin{align*}
-\alpha(n,d)+\frac{1}{4t}\Cc-\frac{1}{4}\partial_{t}\Cc& =\frac{1}{4}\left(-\sum_{i=2}^{n}x_i\cdot \nabla_{x_{i-1}}\Cc+\frac{1}{4t}|\nabla_{x_{n}}\Cc|^{2}-\Delta_{n}\Cc\right).
\end{align*}
By re-arranging the terms in the above equality and recalling that $\alpha(n,d)=\frac{n^2d}{2}$, we obtain \eqref{eq: eqn for C}. This finishes the proof of the proposition.
\end{proof}
The following matrices will play important role in the rest of the paper
\begin{align}
&H_{1}(t):=\mathrm{diag}(1,t,\cdots,t^{n-1})\quad \text{and}\quad H_{2}(t):=\begin{pmatrix}1 & t & \frac{t^{2}}{2!} & \frac{t^{3}}{3!} & \cdots & \frac{t^{n-1}}{(n-1)!}\\
 & t & t^{2} & \frac{t^{3}}{2!} & \cdots & \frac{t^{n-1}}{(n-2)!}\\
 &  & t^{2} & \frac{t^{3}}{1!} & \cdots & \frac{t^{n-1}}{(n-3)!}\\
 &  &  & \ddots & \cdots & \vdots\\
 &  &  &  &  & t^{n-1}
\end{pmatrix},\label{eq: H1 and H2}
\\&Q:=\begin{pmatrix}0\\
1 & 0\\
 & 1 & 0\\
 &  & \ddots & \ddots\\
 &  &  & 1 & 0
\end{pmatrix} \quad\text{and}\quad D:=\mathrm{diag}(0,\ldots,0,1),\label{Q and D}
\\&T_1:=(2n-1)H_{1}^{T}MH_{1}-2t(H_{1}')^{T}MH_{1}-t^{2-2n}H_{1}^{T}MH_{2}DH_{2}^{T}MH_{1},\label{T1}
\\&T_{2}:=(1-2n)H_{2}^{T}MH_{1}+t\big((H_{2}')^{T}MH_{1}+H_{2}^{T}\,M\,H_{1}'\big)-tQH_{2}^{T}MH_{1}+H_{2}^{T}MH_{0}MH_{1}\label{eq: T2}
\\&T_3:=(2n-1)H_{2}^{T}MH_{2}-2t(H_{2}')^{T}MH_{2}+2tQH_{2}^{T}MH_{2}-t^{2-2n}H_{2}^{T}MH_{2}DH_{2}^{T}MH_{2}\label{eq: T3}
\end{align}
Note that $H_1(t), H_2(t), Q,D\in \R^{dn\times dn}$. Each entry of these matrices should be understood as a matrix of order $d$ that equals to the entry multiply with the $d$-dimensional identity matrix.
\begin{proposition}
\label{prop: aux}
 The following assertions hold
\begin{enumerate}[(1)]
\item $T_1$ is anti-symmetric.
\item $T_2=0$.
\item $T_3$ is anti-symmetric.
\item $\mathrm{Tr}(DH_{2}^{T}MH_{2})=n^2\,d\,t^{2(n-1)}$.
\end{enumerate}
\end{proposition}
\begin{proof}
The assertions of the lemma are proved by using combinatorial techniques and are given in Appendix \ref{sec: App}.
\end{proof}
The following lemma is elementary but will be used several times in the sequel. We include it for the sake of completeness and reference.
\begin{lemma}
\label{lemma: anti}
 Suppose that $A\in\R^{N\times N}$ is an anti-symmetric
matrix, then for all $x\in\R^{N}$ we have,
\begin{equation*}
x^{T}Ax=0.
\end{equation*}
\end{lemma} 
\begin{proof} We have $x^{T}Ax=x^{T}A^{T}x=\frac{1}{2}x^{T}(A+A^{T})x=0$.
\end{proof}
We are now ready to prove Theorem \ref{theo: maintheorem1}.
\begin{proof}[\textbf{Proof of Theorem \ref{theo: maintheorem1}}]
To prove Theorem \ref{theo: maintheorem1}, by Proposition \ref{prop: Phi and C} it is sufficient to prove \eqref{eq: eqn for C}.
According to Theorem \ref{theo: C} we have
\begin{align*}
\Cc_t(\x,\y)=\,t^{2-2n}\,[\b(t,\x,\y)]^{T}M\b(t,\x,\y)
\end{align*}
where 
\begin{align*}
\b(t,\x,\y) &=\begin{pmatrix}y_{1}-x_{1}-\frac{t}{1}x_{2}-...-\frac{t^{n-1}}{(n-1)!}x_{n}\\
\vdots\\
t^{i-1}\Big(y_{i}-\sum_{j=i}^{n}\frac{t^{j-i}}{(j-i)!}x_{j}\Big)\\
\vdots\\
t^{n-1}(y_{n}-x_{n})
\end{pmatrix}\\
 & =\begin{pmatrix}y_{1}\\
ty_{2}\\
t^{2}y_{3}\\
\vdots\\
t^{n-1}y_{n}
\end{pmatrix}-\begin{pmatrix}1 & t & \frac{t^{2}}{2!} & \frac{t^{3}}{3!} & \cdots & \frac{t^{n-1}}{(n-1)!}\\
 & t & t^{2} & \frac{t^{3}}{2!} & \cdots & \frac{t^{n-1}}{(n-2)!}\\
 &  & t^{2} & \frac{t^{3}}{1!} & \cdots & \frac{t^{n-1}}{(n-3)!}\\
 &  &  & \ddots & \cdots & \vdots\\
 &  &  &  &  & t^{n-1}
\end{pmatrix}\begin{pmatrix}x_{1}\\
x_{2}\\
x_{3}\\
\vdots\\
x_{n}
\end{pmatrix}\\
 & =H_{1}(t)\y-H_{2}(t)\x,
\end{align*}
where $H_1$ and $H_2$ are given in \eqref{eq: H1 and H2}. Therefore, we have
\begin{align}
\C_t(\x,\y) & =t^{2-2n}[\y^{T}H_{1}^{T}-\x^{T}H_{2}^{T}]\,M\,[H_{1}\y-H_{2}\x]\nonumber\\
 & =t^{2-2n}\Big[\y^{T}H_{1}^{T}\,M\,H_{1}\y-\x^{T}H_{2}^{T}\,M\,H_{1}\y-\y^{T}H_{1}^{T}\,M\,H_{2}\x+\x^{T}H_{2}^{T}MH_{2}\x\Big]\nonumber\\
 & =t^{2-2n}\Big[\y^{T}H_{1}^{T}\,M\,H_{1}\y-2\x^{T}H_{2}^{T}\,M\,H_{1}\y+\x^{T}H_{2}^{T}MH_{2}\x\Big].\label{eq: formula C}
\end{align}
Next we will verify Proposition \ref{prop: Phi and C}. We need to show that the function $\Cc$ satisfies Equation \eqref{eq: eqn for C}.

We compute the time-derivative of $\Cc$ first.
\begin{align}
\label{eq:verify1}
\partial_{t}\Cc & =\y^{T}\Big[(2-2n)t^{1-2n}H_{1}^{T}MH_{1}+2t^{2-2n}(H_{1}')^{T}MH_{1}\Big]\y\nonumber\\
 & \qquad-2\x^{T}\Big[(2-2n)t^{1-2n}H_{2}^{T}MH_{1}+t^{2-2n}\big((H_{2}')^{T}MH_{1}+H_{2}^{T}\,M\,H_{1}'\big)\Big]\y\nonumber\\
 & \qquad+\x^{T}\Big[(2-2n)t^{1-2n}H_{2}^{T}MH_{2}+2t^{2-2n}(H_{2}')^{T}MH_{2}\Big]\x\nonumber\\&=t^{1-2n}\Bigg\{
 \y^{T}\Big[(2-2n)H_{1}^{T}MH_{1}+2t(H_{1}')^{T}MH_{1}\Big]\y\nonumber\\
 & \qquad-2\x^{T}\Big[(2-2n)H_{2}^{T}MH_{1}+t\big((H_{2}')^{T}MH_{1}+H_{2}^{T}\,M\,H_{1}'\big)\Big]\y\nonumber\\
 & \qquad+\x^{T}\Big[(2-2n)H_{2}^{T}MH_{2}+2t(H_{2}')^{T}MH_{2}\Big]\x\Bigg\}.
\end{align}
From \eqref{eq: formula C}, we have
\begin{align*}
\nabla_{\x}\Cc & =2t^{2-2n}\Big[H_{2}^{T}\,M\,H_{2}\,\x-H_{2}^{T}\,M\,H_{1}\,\y\Big].
\end{align*}
Using the matrices $Q$ and $D$ defined in \eqref{Q and D}, $\sum_{i=2}^{n}x_{i}\cdot\nabla_{x_{i-1}}\Cc$ and $\nabla_{x_n}\Cc$ can be computed as follows
\begin{align}
\sum_{i=2}^{n}x_{i}\cdot\nabla_{x_{i-1}}\Cc & =\x^{T}Q\nabla_{\x}\Cc=2t^{2-2n}\Big[\x^{T}QH_{2}^{T}MH_{2}\x-\x^{T}Q\,H_{2}^{T}\,M\,H_{1}\y\Big],\label{eq: verify2}
\end{align}
and similarly
\begin{align*}
\nabla_{x_{n}}\Cc=D\nabla_{\x}C=2t^{2-2n}D\Big[H_{2}^{T}\,M\,H_{2}\,\x-H_{2}^{T}\,M\,H_{1}\,\y\Big].
\end{align*}
Therefore, we get
\begin{align}
|\nabla_{x_{n}}\Cc|^{2} & =4t^{4-4n}[\x^{T}H_{2}^{T}MH_{2}-\y^{T}H_{1}^{T}MH_{2}]D^{T}D\Big[H_{2}^{T}\,M\,H_{2}\,\x-H_{2}^{T}\,M\,H_{1}\,\y\Big]\nonumber\\
 & =4t^{4-4n}[\x^{T}H_{2}^{T}MH_{2}-\y^{T}H_{1}^{T}MH_{2}]D\Big[H_{2}^{T}\,M\,H_{2}\,\x-H_{2}^{T}\,M\,H_{1}\,\y\Big],\label{eq: verify3}
\end{align}
where to obtain the second equality we have used the fact that $D^{T}D=\mathrm{diag}(0,\ldots,0,1)=D$.

The Laplacian $\Delta_{x_n} \Cc$ is then computed via the Trace operator.
\begin{align}
\Delta_{x_{n}}\Cc=2t^{2-2n}\mathrm{Tr}(DH_{2}^{T}MH_{2}).\label{eq: vefify4}
\end{align}
We now verify that
\[
\partial_{t}\Cc=\frac{\Cc}{t}+\sum_{i=2}^{n}x_{i}\cdot\nabla_{x_{i-1}}\Cc-\frac{1}{4t}|\nabla_{x_n}\Cc|^{2}+\Delta_{x_n}\Cc-2dn^2.
\]
Substituting the computations from \eqref{eq:verify1} to \eqref{eq: vefify4}, we need to verify that
\begin{align*}
2dn^2\,t^{2n-1}& =\Big[\y^{T}\big(H_{1}^{T}MH_{1}\big)\y-2\x^{T}H_{2}^{T}MH_{1}\y+\x^{T}H_{2}^{T}MH_{2}\x\Big]\\
 & -\Big(\y^{T}\Big[(2-2n)H_{1}^{T}MH_{1}+2t(H_{1}')^{T}MH_{1}\Big]\y\\
 & \qquad-2\x^{T}\Big[(2-2n)H_{2}^{T}MH_{1}+t\big((H_{2}')^{T}MH_{1}+H_{2}^{T}\,M\,H_{1}'\big)\Big]\y\\
 & \qquad+\x^{T}\Big[(2-2n)H_{2}^{T}MH_{2}+2t(H_{2}')^{T}MH_{2}\Big]\x\Big)\\
 & +2t\Big[\x^{T}QH_{2}^{T}MH_{2}\x-\x^{T}Q\,H_{2}^{T}\,M\,H_{1}\y\Big]\\
 & -t^{2-2n}[\x^{T}H_{2}^{T}MH_{2}-\y^{T}H_{1}^{T}MH_{2}]D\Big[H_{2}^{T}\,M\,H_{2}\,\x-H_{2}^{T}\,M\,H_{1}\,\y\Big]\\
 & +2t\mathrm{Tr}(DH_{2}^{T}MH_{2}),
\end{align*}
or equivalently, using \eqref{T1}--\eqref{eq: T3}, we need to verify that
\begin{align}
\label{eq: verify5}
2 dn^2t^{2n-1} & =\y^{T}\Big[(2n-1)H_{1}^{T}MH_{1}-2t(H_{1}')^{T}MH_{1}-t^{2-2n}H_{1}^{T}MH_{2}DH_{2}^{T}MH_{1}\Big]\y\nonumber\\
 & \quad+2\x^{T}\Big[(1-2n)H_{2}^{T}MH_{1}+t\big((H_{2}')^{T}MH_{1}+H_{2}^{T}\,M\,H_{1}'\big)-tQH_{2}^{T}MH_{1}\nonumber\\
 & \qquad\qquad\qquad+t^{2-2n}H_{2}^{T}MH_{2}DH_{2}^{T}MH_{1}\Big]\y\nonumber\\
 & \quad+\x^{T}\Big[(2n-1)H_{2}^{T}MH_{2}-2t(H_{2}')^{T}MH_{2}+2tQH_{2}^{T}MH_{2}-t^{2-2n}H_{2}^{T}MH_{2}DH_{2}^{T}MH_{2}\Big]\x\nonumber\\
 & \quad+2t\mathrm{Tr}(DH_{2}^{T}MH_{2})\nonumber
 \\&=\y^T T_1 \y+2\x^T T_2\y+ \x^T T_3\x+2t\mathrm{Tr}(DH_{2}^{T}MH_{2}).
\end{align}
According to Proposition \ref{prop: aux} we have 
\begin{enumerate}[(i)]
\item $T_2=0$ and $2t\mathrm{Tr}(DH_{2}^{T}MH_{2})=2t\, dn^2 t^{2(n-1)}= 2 dn^2 t^{2n-1}$, 
\item $T_1$ and $T_3$ are anti-symmetric. By Lemma \ref{lemma: anti},
we have $\y^T T_1 \y=0=\x^T T_3\x$.
\end{enumerate}
Therefore, using (i)--(ii) above, we obtain
\[
\y^T T_1 \y+2\x^T T_2\y+ \x^T T_3\x+2t\mathrm{Tr}(DH_{2}^{T}MH_{2})=2dn^2t^{2n-1}
\]
which is equal to the left-hand side of \eqref{eq: verify5} as required.

Finally, the initial condition \eqref{eq: initial diract} follows from the representation of $\Cc$ in Theorem \ref{theo: C} and the formula of $\b$ in \eqref{eq: b}.
\end{proof}
\section{The variational formulation of Equation \eqref{eq: gKramers}}
\label{sec: variatonal}
\subsection{Well-posedness of Scheme \ref{scheme} and the Euler-Langrange equation}
In this section we prove the well-posedness of Scheme \ref{scheme} and establish the Euler-Lagrange equations for the sequence of its minimizers. 

Under Assumption \ref{assumpt} the free energy functional
\[
\F(\rho):=\int_{\R^{dn}}\big(V(x_n)+\log\rho\big)\rho\,d\x
\]
is well-defined in $\P_2(\R^{dn})$. The following two lemmas show that Scheme \ref{scheme} is well-defined. Their proofs are now classical, see e.g.,~\cite[Theorem 1.3]{Vil03},  \cite[Proposition 4.1]{JKO98}, and \cite[Lemma 4.2]{Hua00}). Hence we will omit them here. 
\begin{lemma}\label{existoptimalplan}
Let $\rho_0,\rho\in \P_2(\R^{dn})$ be given. There exists a unique optimal plan $\PSopt\in \Gamma(\rho_0,\rho)$ such that
\emph{
\begin{equation}
\W_h(\rho_0,\rho)=\int_{\R^{dn}\times \R^{dn}}\Cc_h(\x,\y)\PSopt(d\x\, d\y).
\end{equation}}
\end{lemma}
\begin{lemma}\label{wellposedness}
Let $\rho_0\in\P_2(\R^{dn})$ be given. If $h$ is small enough, then the minimization problem
\begin{equation}
\min_{\rho\in\P_2(\R^{dn})}\frac{1}{2h} \W_h(\rho_0,\rho)+ \F(\rho),
\end{equation}
has a unique solution.
\end{lemma}
\noindent
Next we establish the Euler-Lagrange equation for the sequence of minimizers of Scheme \ref{scheme}. We will need two auxiliary lemmas whose proofs are presented in Appendix \ref{sec: App}.
\begin{lemma}
\label{lem: InverH2H1}
$H_2^{-1}H_1=H$ where
\begin{equation}
\label{eq: H}
H_{ij}=\begin{cases}
0, \quad \text{if}~ j<i\\
(-1)^{j-i}\frac{h^{j-i}}{(j-i)!},\quad \text{if}~ j\geq i.
\end{cases}
\end{equation}
In particular $H_{ii}=1, \quad H_{ii+1}=-h$ and $H_{ij}=o(h^2)$ for $j\geq i+2$. Note that $H\in\R^{dn \times dn}$ where $H_{ij}$ should be understood as $H_{ij}I_d$.
\end{lemma}
\begin{lemma}
\label{lem:K}
Let $\K=h^{2n-2}(H_2^TMH_1)^{-1}$. Then
\begin{equation}
\K_{ij}=(-1)^{n-j}\frac{h^{2n-i-j}}{(2n-i-j+1)!}.
\end{equation}
In particular, $\K_{nn}=1$ and $\K_{ij}=o(h)$ for all $(i,j)\neq (n,n)$. Note also that $\K\in \R^{dn\times dn}$ where $\K_{ij}$ should be understood as $\K_{ij}I_d$.
\end{lemma}
Having these two lemmas, we are now ready to derive the Euler-Lagrange equation for the sequence of minimizers in Scheme \ref{scheme}.
\begin{lemma}[Euler-Lagrange equation for the sequence of minimizers]
\label{lem: EU eqn}
Let $\{\rho_k^h\}_{k\geq 1}$ be the sequence of the minimizers of Scheme \ref{scheme}. Then we have
\emph{
\begin{multline}
\label{eq: ELeqn}
\frac{1}{h}\int_{\R^{2dn}}(\y-\x)\cdot\nabla \varphi(\y)P^k_h(d\x\,d\y)-\sum_{i=2}^{n}\int_{\R^{dn}}x_i\cdot\nabla_{x_{i-1}}\varphi(\x)\rho_k^h(\x) d\x + \int_{\R^{dn}}\nabla_{x_n}V(x_n)\cdot\nabla_{x_n}\varphi(\x)\rho_k^h(\x)\, d\x\\-\int_{\R^{dn}}\Delta_{x_n}\varphi(\x)\,\rho_k^h(\x)d\x=o(h)\qquad \text{for all}~~\varphi\in C_0^\infty(\R^{dn},\R),
\end{multline}}
where $P_k^h$ is the optimal plan in $\W_h(\rho_{k-1}^h,\rho_k^h)$.
\end{lemma}
\begin{proof}
Let $\overline\mu\in\P_2(\R^{dn})$ be given and let $\mu$ be the unique solution of the minimization problem
\begin{equation*}
\min_{\gamma\in\P_2(\R^{dn})}\frac{1}{2 h}\W_h(\overline\mu,\gamma)+\F(\gamma).
\end{equation*}
We will show that
\begin{multline}
\label{eq: EU eqn1}
\frac{1}{h}\int_{\R^{2dn}}(\y-\x)\cdot\nabla \varphi(\y)\PSopt(d\x\,d\y)-\sum_{i=2}^{n}\int_{\R^{dn}}x_i\cdot\nabla_{x_{i-1}}\varphi(\x)\mu(\x)\,d\x + \int_{\R^{dn}}\nabla_{x_n}V(x_n)\cdot\nabla_{x_n}\varphi(\x)\mu(\x)\, d\x\\-\int_{\R^{dn}}\Delta_{x_n}\varphi(\x)\,\mu(\x)d\x=o(h)\qquad \text{for all}~~\varphi\in C_0^\infty(\R^{dn},\R),
\end{multline}
where $\PSopt$ is the optimal plan in $\W_h(\overline\mu,\mu)$.

Although establishing the Euler-Lagrange equation for the minimizer $\mu$ has become a well-established route, see e.g., \cite{JKO98} and \cite{Hua00, DPZ13b} for that of the Fokker-Planck equation ($n=1$)
and of the Kramers equation ($n=2$) respectively, there is one additional difficulty. That is how to select an appropriate perturbation flow from that the Euler-Langrange equation is deduced. We first define a perturbation of $\mu$ by a push-forward under an appropriate flow.
Let $\phi_1,\ldots,\phi_n\in C_0^\infty(\R^{dn},\R^d)$. We define the flows
$\Phi^1,\ldots,\Phi^n\colon[0,\infty)\times\R^{dn}\rightarrow\R^d$ such that
\begin{equation*}
\frac{\partial\Phi^i_s}{\partial s}=\phi_i(\Phi^1_s,\ldots,\Phi^n_s),\quad \Phi^i_0(x_1,\ldots,x_n)=x_i.
\end{equation*}
Let $\mu_s(\x)$ be the push forward of $\mu(\x)$ under the flow $(\Phi^1_s,\ldots,\Phi^n_s)$, i.e., for any $\varphi\in C_0^\infty(\R^{dn},\R)$ we have
\begin{equation}
\int_{\R^{dn}}\varphi(\x)\mu_s(\x)d\x=\int_{\R^{dn}}\varphi(\Phi^1_s(\x),\ldots,\Phi^n_s(\x))\mu(\x)d\x. \label{pushforward}
\end{equation}
Since $(\Phi^1_0,\ldots,\Phi^n_0)=\x$, we have $\mu_0(\x)=\mu(\x)$. Taking derivatives with respect to $s$ of both sides gives
\begin{equation}
\partial_s\mu_s\big|_{s=0} =-\sum_{i=1}^n\text{div}_{x_i}\big(\mu\phi_i\big) \qquad \text{in the sense of distributions}.
\end{equation}
We then compute, using \eqref{eq: initial diract} and \eqref{eq: MK cost}, the stationarity condition for $\mu$ following the calculations in ~\cite{JKO98,Hua00, DPZ13b}
\begin{multline}
\frac{1}{2h}\int_{\R^{2dn}}\Big[\sum_{i=1}^n\nabla_{y_i}\Cc_h(\x,\y)\cdot\phi_i(\y)\Big]\PSopt(d\x\,d\y)+\int_{\R^{dn}}\nabla_{x_n}V(x_n)\cdot\phi_n(\x)\mu(\x)\, d\x\\-\int_{\R^{dn}}\Big[\sum_{i=1}^n\div_{x_i}\phi_i(\x)\Big]\,\mu(\x)d\x=0.
\label{EuLageqn}
\end{multline}
According to \eqref{eq: formula C} we have
\[
\Cc_h(\x,\y)=h^{2-2n}[\y^TH_1^T M H_1\y-2\x^TH_2^TMH_1\y+\x^TH_2^TMH_2\x].
\]
Therefore,
\begin{equation}
\nabla_{\y} \Cc_h(\x,\y)=2h^{2-2n}H_1^TM(H_1\y-H_2\x).
\end{equation}
Let $\varphi\in C_0^\infty(\R^{dn},\R)$. We choose $\phi_1,\ldots, \phi_n$ such that
\begin{equation}
\begin{pmatrix}
\phi_1\\
\vdots\\
\phi_n
\end{pmatrix}=\K\begin{pmatrix}
\nabla_{y_1} \varphi\\
\vdots\\
\nabla_{y_n}\varphi
\end{pmatrix} =\K \nabla \varphi
\end{equation}
where $\K$ is given in Lemma \ref{lem:K} that implies that $h^{2-2n}\K^T (H_1^TMH_2)=I$.  

Using Lemmas \ref{lem: InverH2H1} and \ref{lem:K},  we compute
 \begin{align*}
\text{(i)}~ \sum_{i=1}^n\nabla_{y_i}\Cc_h(\x,\y)\cdot\phi_i(y)&= \nabla_{\y}\Cc_h(\x,\y)\cdot\K\nabla\varphi(\y)=2\Big[(H_2^{-1}H_1-I)\y\cdot \nabla \varphi(\y) +(\y-\x)\cdot\nabla\varphi(\y)\Big]
\\&=2(\y-\x)\cdot\nabla\varphi(\y)-2h\sum_{i=2}^{n} y_{i}\cdot\nabla_{y_{i-1}}\varphi+o(h^2)
\end{align*}
(ii) $\phi_n(\x)=\sum_{j=1}^n \K_{nj}\nabla_{x_j}\varphi(\x)=\nabla_{x_n}\varphi(\x)+o(h)$.
\\ \ \\
(iii) $\sum_{i=1}^n\div_{x_i}\phi_i(\x)=\div_{\x}[\K \nabla \varphi(\x)]=\sum_{i,j}\K_{ij}\partial^2_{x_ix_j}\varphi=\Delta_{x_n}\varphi(\x) +o(h)$.
\\ \ \\
Substituting these calculations back into \eqref{EuLageqn} we obtain 
\begin{multline*}
\frac{1}{h}\int_{\R^{2dn}}(\y-\x)\cdot\nabla \varphi(\y)\PSopt(d\x\,d\y)-\sum_{i=2}^{n}\int_{\R^{dn}}x_i\cdot\nabla_{x_{i-1}}\varphi(\x)\mu(\x)\,d\x+\int_{\R^{dn}}\nabla_{x_n}V(x_n)\cdot\nabla_{x_n}\varphi(\x)\mu(\x)\, d\x\\-\int_{\R^{dn}}\Delta_{x_n}\varphi(\x)\,\mu(\x)d\x=o(h),
\end{multline*}
which is the desired equality \eqref{eq: EU eqn1}.

Applying  \eqref{eq: EU eqn1} for the minimizers $\{\rho_k^h\}_{k\geq 1}$ of Scheme \ref{scheme} yields the statement of Lemma \ref{lem: EU eqn}.
\end{proof}
\subsection{A priori estimates}
In this section, we derive a priori estimates for the sequence of the minimizers of Scheme \ref{scheme}. The proofs of Lemma \ref{priorbound1:scheme}, Lemma \ref{priorbound2} and Lemma \ref{lemma:priorbound3:Scheme2a} below are now standard, see e.g. \cite{JKO98,Hua00,DPZ13b}, hence we omit them.

The following lemma provides an upper bound for the sum $\sum_{k=1}^n\W_h(\rho_{k-1}^h,\rho_k^h)$. From now on, we denote by $M_2(\rho)$ the second moment of a probability measure $\rho$, i.e., $M_2(\rho)=\int |\x|^2\,d \rho$.
\begin{lemma}\label{priorbound1:scheme}
Let $\{\rho_k^h\}_{k\geq1}$ be the sequence of the minimizers of Scheme \ref{scheme} for fixed $h>0$.  Then for any positive integer $n$ and sufficiently small $h$, we have
\begin{equation}
\sum_{k=1}^n\W_h(\rho_{k-1}^h,\rho_k^h)\leq
2 h(\F(\rho_0)-\F(\rho_n^h))+Ch^2\sum_{k=0}^{n}M_2(\rho_k^h)+Cnh^2,\label{sumIneqn:scheme2a}
\end{equation}
for some constant $C>0$ independent of $n$.
\end{lemma}
The next lemma shows boundedness of the second moment $M_2(\rho_k^h)$ and the entropy $S(\rho_k^h)$ locally in time.
\begin{lemma}\label{priorbound2}
There exist positive constants $T_0$, $h_0$, and $C$,  independent of the initial data, such that for any $0<h\leq h_0$, the solutions $\{\rho_k^h\}_{k\geq 1}$ for Scheme \ref{scheme} satisfy
\begin{equation}
M_2(\rho_k^h)\leq C\big[M_2(\rho_0)+1\big] ~ \text{ and } ~|S(\rho_k^h)|\leq C\big[S(\rho_0) + M_2(\rho_0) + 1\big] ~ \text{for any }k\leq K_0,
\end{equation}
where $K_0=\lceil{T_0}/{h}\rceil$.
\end{lemma}
The last lemma of this section extends Lemma \ref{priorbound2} to any final time $T>0$.
\begin{lemma}
\label{lemma:priorbound3:Scheme2a}
Let $\{\rho_k^h\}_{k\geq1}$ be the sequence of the minimizers of Scheme \ref{scheme} for fixed $h>0$. For any $T>0$, there exists a constant $C>0$ depending  on $T$ and on the initial data such that
\begin{equation}
M_2(\rho_k^h)\leq C,\label{2moment}
\end{equation}
\begin{equation}
\sum_{i=1}^k \W_h(\rho_{i-1}^h,\rho_i^h)\leq Ch, \label{summetric}
\end{equation}
\begin{equation}
\int_{\R^{dn}}\max\{\rho_k^h\log \rho_k^h, 0\}\,dqdp\leq C, \label{absoluteentropy}
\end{equation}
for any $h\leq h_0$ and $k\leq K_h$, where $K_h=\left\lceil\frac{T}{h}\right\rceil$.
\end{lemma}
\subsection{Proof of Theorem \ref{theo:maintheorem 2}}
Having established the Euler-Lagrange equation and a priori estimates in previous sections, in this section we prove Theorem \ref{theo:maintheorem 2}. The proof is similar to that of \cite{JKO98,Hua00,DPZ13b}. Therefore, we only present the part that is different, that is to prove the convergence of the discrete Euler-Lagrange equations to the weak formulation \eqref{weakKReqn} of Equation \eqref{eq: gKramers} as $h\to0$. The key point  is to link the Euler-Lagrange equation for the sequence of minimizers obtained in Lemma \ref{lem: EU eqn} to a time-discretization of Equation \eqref{eq: gKramers}.

Let $T>0$ be a given final time. For each $h>0$ we set $K_h \colonequals \lceil T/h\rceil$. Let $(\rho_k^h)_{k\geq 1}$ be the sequence of minimizers of Scheme \ref{scheme} and let $t\mapsto \rho^h(t)$ be the piecewise-constant interpolation~\eqref{interpolation}. By Lemma~\ref{lemma:priorbound3:Scheme2a} we have
\begin{align}
M_2(\rho^h(t))+\int_{\R^{dn}}\max\{\rho^h(t)\log \rho^h(t), 0\}\,d\x&\leq C,\qquad\text{for all} \quad 0\leq t \leq T. \label{sumM2Entropy}
\end{align}
Since the function $z\mapsto \max\{z\log z,0\}$ has super-linear growth, \eqref{sumM2Entropy} guarantees that there exists a subsequence, denoted again by $\rho^h$, and a function $\rho\in L^1((0,T)\times\R^{dn})$ such that
\begin{equation}
\rho^h\rightarrow \rho ~\text{ weakly in }~ L^1((0,T)\times\R^{dn}).
\end{equation}
We now prove that the limit $\rho$ satisfies the weak formulation \eqref{weakKReqn}.
Let $\varphi\in C_c^\infty((-\infty,T)\times \R^{dn})$ be given. All constants~$C$ below depend on the parameters of the problem, on the initial datum $\rho_0$, and on~$\varphi$, but are independent of $k$ and of $h$. 

Let $P_k^{h}\in\Gamma(\rho_{k-1}^h,\rho_k^h)$ be the optimal plan for $\W_h(\rho_{k-1}^h,\rho_k^h)$. For any $0<t<T$, we have
\begin{align}
&\int_{\R^{dn}}\big[\rho_k^h(\x)-\rho_{k-1}^h(\x)\big]\,\varphi(t,\x)d\x=\int_{\R^{dn}}\rho_k^h(\y)\varphi(t,\y)d\y-\int_{\R^{dn}}\rho_{k-1}^h(\x)\varphi(t,\x)d\x\nonumber\\
&\qquad=\int_{\R^{2dn}}\big[\varphi(t,\y)-\varphi(t,\x)\big]\,P_k^{h}(d\x d\y) \nonumber
\\&\qquad=\int_{\R^{2dn}}(\y-\x)\cdot\nabla \varphi(t,\y)P_k^{h}(d\x d\y) +\varepsilon_k,\label{timederivativeappro}
\end{align}
where the error term $\varepsilon_k$ comes from the Taylor expansion of $\varphi$ and can be estimated by
\begin{align}
|\varepsilon_k|&\leq C\int_{\R^{2dn}}|\y-\x|^2\,P_k^{h}(d\x d\y)\notag\\
&\stackrel{\eqref{ineqs:qpC}}\leq C \W_h(\rho_{k-1}^h,\rho_k^h) + Ch^2\big[M_2(\rho_{k-1}^h) + M_2(\rho_k^h)\big]\notag\\
&\stackrel{\eqref{sumM2Entropy}}\leq C \W_h(\rho_{k-1}^h,\rho_k^h) + Ch^2.
\label{epsilon}
\end{align}
Multiplying~\eqref{timederivativeappro} with $\frac{1}{h}$ and combining with~\eqref{eq: ELeqn} we get
\begin{multline}
\int_{\R^{dn}}\left(\frac{\rho_k^h(t,\x)-\rho_{k-1}^h(t,\x)}{h}\right)\varphi(t,\x)d\x
\\
=\int_{\R^{dn}}\Big[\sum_{i=2}^{n}x_i\cdot\nabla_{x_{i-1}}\varphi(\x) -\nabla_{x_n}V(x_n)\cdot\nabla_{x_n}\varphi(\x)+\Delta_{x_n}\varphi(\x)\Big]\,\rho_k^h(\x)d\x+\theta_k(t), \label{approximation1}
\end{multline}
where
\begin{equation}
|\theta_k(t)|\leq\frac{|\varepsilon_k|}{h}+o(h)\stackrel{\eqref{sumM2Entropy},\eqref{epsilon}}\leq
\frac Ch \W_h(\rho_{k-1}^h,\rho_k^h)  + o(h).
\label{thetaform}
\end{equation}
It is worthy noting that $\theta_k$ depends on $t$ through the $t$-dependence of $\varphi$. Integrating (\ref{approximation1}) with respect to $t$ from $(k-1)h$ to $kh$, we obtain
\begin{align}
&\int_{(k-1)h}^{kh}\int_{\R^{dn}}\left(\frac{\rho_k^h(\x)-\rho_{k-1}^h(\x)}{h}\right)\varphi(t,\x)d\x dt\nonumber
\\&\quad=\int_{(k-1)h}^{kh}\int_{\R^{dn}}\Big[\sum_{i=2}^{n}x_i\cdot\nabla_{x_{i-1}}\varphi(\x) -\nabla_{x_n}V(x_n)\cdot\nabla_{x_n}\varphi(\x)+\Delta_{x_n}\varphi(\x)\Big]\,\rho_k^h(\x)d\x\,dt+ \int_{(k-1)h}^{kh}\theta_k(t)dt\nonumber
\\&\quad=\int_{(k-1)h}^{kh}\int_{\R^{dn}}\Big[\sum_{i=2}^{n}x_i\cdot\nabla_{x_{i-1}}\varphi(\x) -\nabla_{x_n}V(x_n)\cdot\nabla_{x_n}\varphi(\x)+\Delta_{x_n}\varphi(\x)\Big]\,\rho^h(t,\x)d\x\,dt+ \int_{(k-1)h}^{kh}\theta_k(t)dt\nonumber.
\end{align}
Summing this relation from $k=1$ to $K_h$ gives
\begin{align}
&\sum_{k=1}^{K_h}\int_{(k-1)h}^{kh}\int_{\R^{dn}}\left(\frac{\rho_k^h(\x)-\rho_{k-1}^h(\x)}{h}\right)\varphi(t,\x)d\x dt\nonumber
\\&\quad=\int_0^T\int_{\R^{dn}}\Big[\sum_{i=2}^{n}x_i\cdot\nabla_{x_{i-1}}\varphi(\x) -\nabla_{x_n}V(x_n)\cdot\nabla_{x_n}\varphi(\x)+\Delta_{x_n}\varphi(\x)\Big]\,\rho^h(t,\x)d\x\,dt+ R_h,\label{approximationsum}
\end{align}
where
\begin{equation}
R_h=\sum_{k=1}^{K_h}\int_{(k-1)h}^{kh}\theta_k(t)dt.\label{Rh}
\end{equation}
By a discrete integration by parts, the left hand side of~\eqref{approximationsum} can be written as
\begin{align}
&-\int_0^h\int_{\R^{dn}}\rho_0(\x)\frac{\varphi(t,\x)}{h}d\x dt+\int_0^T\int_{\R^{dn}}\rho^h(t,\x)\left(\frac{\varphi(t,\x)-\varphi(t+h,\x)}{h}\right)d\x dt\label{approximationsumFINAL}.
\end{align}
From~\eqref{approximationsum} and~\eqref{approximationsumFINAL} we obtain
\begin{align}
&\int_0^T\int_{\R^{dn}}\rho^h(t,\x)\left(\frac{\varphi(t,\x)-\varphi(t+h,\x)}{h}\right)d\x dt\nonumber
\\&=\int_0^T\int_{\R^{dn}}\Big[\sum_{i=2}^{n}x_i\cdot\nabla_{x_{i-1}}\varphi(\x) -\nabla_{x_n}V(x_n)\cdot\nabla_{x_n}\varphi(\x)+\Delta_{x_n}\varphi(\x)\Big]\,\rho^h(t,\x)d\x\,dt\nonumber
\\&\qquad+\int_0^h\int_{\R^{dn}}\rho_0(\x)\frac{\varphi(t,\x)}{h}d\x dt+R_h. \label{weakconvergence}
\end{align}
Next we show that $R_h\rightarrow 0$ as $h\to0$. Indeed, we have
\begin{eqnarray}
|R_h|\stackrel{\eqref{Rh}}\leq \sum_{k=1}^{K_h}\int_{(k-1)h}^{kh}|\theta_k(t)|dt&\stackrel{\eqref{thetaform}}\leq& C\sum_{k=1}^{K_h}\int_{(k-1)h}^{kh}\left(\frac{1}{h}\W_h(\rho_{k-1}^h,\rho_k^h)+h\right)dt\nonumber
\\&=& C\sum_{k=1}^{K_h}\big[\W_h(\rho_{k-1}^h,\rho_k^h)+Ch^2\big]
\stackrel{\eqref{summetric}}\leq Ch.\nonumber
\end{eqnarray}
Taking the limit $h\rightarrow0$ in~\eqref{weakconvergence} yields equation~\eqref{weakKReqn} proving \eqref{weaklyconverence}.

The proof of the stronger convergence~\eqref{pointwiseconvergence} and of the continuity~\eqref{intinialconvergence} at $t=0$ follows from the equi-near-continuity estimate, see \cite{JKO98,Hua00,DPZ13b}
\[
W_2\big(\rho^h(t_1),\rho^h(t_2)\big)^2 \leq C(|t_2-t_1| + h),
\]
where $W_2(\rho_0,\rho_1)$ is the Wasserstein distance between $\rho_0$ and $\rho_1$ defined in \eqref{eq: W2 distance}. This estimate follows from the inequality (see~\eqref{ineqs:qpC})
\[
|\x-\y|^2 \leq C\big[\Cc_h(\x,\y) + h^2(|\x|^2+|\y|^2)\big],
\]
and the estimates~\eqref{sumM2Entropy} and~\eqref{summetric}.
\section{Appendix}
\label{sec: App}
This appendix contains proofs of technical lemmas in the previous sections.
\subsection{Proof of Proposition \ref{prop: aux}}
In this appendix, we prove Proposition \ref{prop: aux}. For the convenience, we recall the relevant matrices here:
\begin{align}
&H_{1}(t):=\mathrm{diag}(1,t,\cdots,t^{n-1})\quad \text{and}\quad H_{2}(t):=\begin{pmatrix}1 & t & \frac{t^{2}}{2!} & \frac{t^{3}}{3!} & \cdots & \frac{t^{n-1}}{(n-1)!}\\
 & t & t^{2} & \frac{t^{3}}{2!} & \cdots & \frac{t^{n-1}}{(n-2)!}\\
 &  & t^{2} & \frac{t^{3}}{1!} & \cdots & \frac{t^{n-1}}{(n-3)!}\\
 &  &  & \ddots & \cdots & \vdots\\
 &  &  &  &  & t^{n-1}
\end{pmatrix},\label{eq: H_1 and H_2}
\\&Q:=\begin{pmatrix}0\\
1 & 0\\
 & 1 & 0\\
 &  & \ddots & \ddots\\
 &  &  & 1 & 0
\end{pmatrix} \quad\text{and}\quad D:=\mathrm{diag}(0,\ldots,0,1).\label{eq: Q and D}
\end{align}

We need several auxiliary lemmas. The first one is an explicit formula for the inverse of the matrix $B$ define in \eqref{B}.

%

%
%
%
%
%
%
%
%

\begin{lemma}
Recall the matrix $B$ in \eqref{B}
\begin{equation*}
B_{ki}=\begin{cases}
(-1)^{n-k}\frac{(n+i-1)!}{(k+i-n-1)!}, & \quad\text{if}\quad k+i\ge n+1\\
0 & \quad\text{if}\quad k+i<n+1.
\end{cases}
\end{equation*}
Then $B^{-1}$ has the following form
\begin{equation}
\label{eq: inverse B}
(B^{-1})_{ki}=\begin{cases}
(-1)^{k-1}\frac{1}{(n+k-1)!(n+1-k-i)!}, &\quad\text{if}\quad k+i\le n+1\\
0 & \quad\text{if}\quad k+i>n+1.
\end{cases}
\end{equation}
\end{lemma}
\begin{proof}
Define the matrix $\tilde{B}$ by
\begin{equation}
\tilde{B}_{ki}=\begin{cases}
(-1)^{k-1}\frac{1}{(n+k-1)!(n+1-k-i)!}, &\quad\text{if}\quad k+i\le n+1\\
0 & \quad\text{if}\quad k+i>n+1.
\end{cases}
\end{equation}
We now show that $B\tilde{B}=I$. We consider three cases
\begin{enumerate}
\item If $k<j$, then 
\[
\sum\limits_{i=1}^{n}B_{ki}\tilde{B}_{ij}=\sum\limits_{i=n+1-k}^{n}B_{ki}\tilde{B}_{ij}=0.
\]
The last equality is $0$ since by the definition of $\tilde{B}$ we have $\tilde{B}_{ij}=0$ for every $i\geq n+1-k>n+1-j$.
\item If $k=j$, then similarly as the case above, we get
\begin{align*}
\sum\limits_{i=1}^{n}B_{ki}\tilde{B}_{ij}&=B_{k,n+1-i}\tilde{B}_{n+1-j,j}
\\&=(-1)^{n-k}(n+i-1)!\times (-1)^{n-j}\frac{1}{(n+i-1)!}=1.
\end{align*}
\item If $k>j$, then since $B_{ki}=0$ for $i<n+1-k$ and $\tilde{B}_{ij}=0$ for $i>n+1-j$, we get
\begin{align*}
\sum_{i=1}^{n}B_{ki}\tilde{B}_{ij}  & =\sum\limits_{\substack{1\leq i\leq n\\i\geq n+1-k\\i\leq n+1-j}}B_{ki}\tilde{B}_{ij}\\
 & =\sum_{i=n+1-k}^{n+1-j}B_{ki}(\tilde{B})_{ij}\\
 & =\sum_{i=n+1-k}^{n+1-j}(-1)^{n-k}\frac{(n+i-1)!}{(k+i-n-1)!}(-1)^{i-1}\frac{1}{(n+i-1)!(n+1-i-j)!}\\
 & =(-1)^{n-1}\sum_{i=n+1-k}^{n+1-j}(-1)^{i-k}\frac{1}{(k+i-n-1)!(n+1-i-j)!}\\
 & =(-1)^{n-1}\sum_{l=0}^{k-j}(-1)^{i-k}\frac{1}{l!(n+1-(l+n+1-k)-j)!}(\text{ change variable: }i=l+n+1-k)\\
 & =(-1)^{n-1}\frac{1}{(k-j)!}\sum_{l=0}^{k-j}(-1)^{l+n+1-k-k}\frac{(k-j)!}{l!(k-j-l)!}\\
 & =(-1)^{n-1}(-1)^{n+1-k-k}\frac{1}{(k-j)!}\sum_{l=0}^{k-j}(-1)^{l}\frac{(k-j)!}{l!(k-j-l)!}\\
 & =(1-1)^{k-j}=0.
\end{align*}
\end{enumerate}
From these three cases, we imply that $B\tilde{B}=I$, therefore $B^{-1}=\tilde{B}$ which concludes the proof of the lemma.
\end{proof}
The next two lemmas show relations between the matrices $H_1, H_2, Q$ and $D$.
\begin{lemma} 
\label{lem: Ho}
It holds that
\begin{equation}
\label{eq: H0}
(H_2DH_2^T)_{ij}=t^{2(n-1)} (H_0)_{ij}, \quad \text{where}\quad (H_0)_{ij}=\frac{1}{(n-i)!(n-j)!}.
\end{equation}
\end{lemma}
\begin{proof}
For three matrices $X,\,Y,\,Z$, we have
\[
(XYZ)_{ij}=\sum\limits _{kl}X_{ik}Y_{kl}Z_{lj}.
\]
In particular, applying this identity for $X=H_2, Y=D$ and $Z=H_2^T$, we get
\begin{align*}
(H_0)_{ij}&=(H_{2}DH_{2}^{T})_{ij} =\sum\limits _{k,l}(H_{2})_{il}D_{kl}(H_{2}^{T})_{lj}
\\&=\sum\limits _{k,l}(H_{2})_{il}D_{kl}(H_{2})_{jl}\\
 & \overset{(*)}{=}(H_{2})_{in}(H_{2})_{jn}
 \\&\overset{(**)}{=}\frac{t^{n-1}}{(n-i)!}\frac{t^{n-1}}{(n-j)!}=\frac{t^{2(n-1)}}{(n-i)!(n-j)!},
\end{align*}
where we used the definition of $D$ and $H_2$ to obtain $(*)$ and $(**)$ respectively. This finishes the proof of the lemma.
\end{proof}
\begin{lemma} 
\label{lem: diag relation}
It holds that
\begin{equation}
[(2n-1)I-2t(H_{2}^{T})^{-1}(H_{2}')^{T}+2t(H_{2}^{T})^{-1}QH_{2}^{T}]=\mathrm{diag}(2n-1,2n-3,\ldots,1).\label{eq34}
\end{equation}
\end{lemma}
\begin{proof}
Let $T_{23}:=[-2t(H_{2}^{T})^{-1}(H_{2}')^{T}+2t(H_{2}^{T})^{-1}QH_{2}^{T}]$. Then
\begin{align}
(T_{23})_{ij} & =-2t\sum_{k=1}^{n}(H_{2}^{T})_{ik}^{-1}(H_{2}')_{kj}^{T}+2t\sum_{k,l}(H_{2}^{T})_{ik}^{-1}Q_{kl}(H_{2}^{T})_{lj}\nonumber \\
 & =-2t\sum_{k=1}^{n}(H_{2}^{T})_{ik}^{-1}(H_{2}')_{kj}^{T}+2t\sum_{l,k:k=l+1}(H_{2}^{T})_{ik}^{-1}(H_{2}^{T})_{lj}\nonumber \\
 & =-2t\sum_{k=1}^{n}(H_{2}^{T})_{ik}^{-1}(H_{2}')_{kj}^{T}+2t\sum_{k=1}^{n-1}(H_{2}^{T})_{i,k+1}^{-1}(H_{2}^{T})_{kj}\nonumber.
\end{align}
Note that $(H_{2}^{T})_{ik}^{-1}(H_{2}')_{kj}^{T}\ne0$ iff $i\ge k,k\ge j$ and
$(H_{2}^{T})_{i,k+1}^{-1}(H_{2}^{T})_{kj}\ne0$ iff $i\ge k+1,k\ge j$. We transform further $(T_{23})_{ij}$.
\begin{enumerate}
\item If $j>i$ then $(H_{2}^{T})_{ik}^{-1}(H_{2}')_{kj}^{T}=(H_{2}^{T})_{i,k+1}^{-1}(H_{2}^{T})_{kj}=0$. This gives that $(T_{23})_{ij} =0$.
\item If $j<i$, then we have
\begin{align*}
(T_{23})_{ij} & =-2t\sum_{k=j}^{i}(H_{2}^{T})_{ik}^{-1}(H_{2}')_{kj}^{T}+2t\sum_{k=j}^{i}(H_{2}^{T})_{i,k+1}^{-1}(H_{2}^{T})_{kj}\\
 & =-2t\sum_{k=j}^{i}\frac{(-1)^{i-k}}{t^{k-1}(i-k)!}\frac{(k-1)t^{k-2}}{(k-j)!}+2t\sum_{k=j}^{i-1}\frac{(-1)^{i-(k+1)}}{t^{k+1-1}(i-k-1)!}\frac{t^{k-1}}{(k-j)!}\\
 & =-2\sum_{k=j}^{i}\frac{(-1)^{i-k}}{(i-k)!}\frac{(k-1)}{(k-j)!}-2\sum_{k=j}^{i-1}\frac{(-1)^{i-k}}{(i-k-1)!}\frac{1}{(k-j)!}\\
 & =-2\sum_{k=j}^{i}\frac{(-1)^{i-k}}{(i-k)!}\frac{(k-1)}{(k-j)!}-2\sum_{k=j}^{i-1}\frac{(-1)^{i-k}}{(i-k-1)!}\frac{1}{(k-j)!}\\
 & =-2\sum_{l=0}^{i-j}\frac{(-1)^{i-j-l}}{(i-j-l)!}\frac{(l+j-1)}{l!}-2\sum_{l=0}^{i-1-j}\frac{(-1)^{i-(j-l)}}{(i-j-l-1)!}\frac{1}{l!}
 \\&=(I)+(II).
\end{align*}
Consider the function $g(x)=(x-1)^{i-j}$. On the one hand, since $g'(x)=(i-j)(x-1)^{i-j}$, it follows that $g'(1)=0$. On the other hand, we have
\begin{equation*}
g'(x) =\sum_{l=0}^{i-j}\frac{lx^{l-1}(-1)^{i-j-l}}{(i-j-l)!l!},
\end{equation*}
which implies that
\begin{equation}
\label{eq: I1}
g'(1)=\sum_{l=0}^{i-j}\frac{l(-1)^{i-j-l}}{(i-j-l)!l!}=0.
\end{equation}
In addition, we have 
\begin{equation}
\label{eq: I2}
(j-1)\sum_{l=0}^{i-j}\frac{(-1)^{i-j-l}}{(i-j-l)!}\frac{1}{l!}=(j-1)(1+(-1))^{i-j}=0.
\end{equation}
Summing up \eqref{eq: I1} and \eqref{eq: I2} yields $(I)=0$. The term $(II)$ is equal to
\[
(II)=\frac{1}{(i-j-1)!}(1-1)^{i-j-1}=0.
\]
Hence if $j<i$, then $(T_{23})_{ij}=0$.
\item Finally, if $i=j$, then we have
\begin{align*}
(T_{23})_{ii} & =-2t(H_{2}^{T})_{ii}^{-1}(H_{2}')_{ii}^{T}+2t(H_{2}^{T})_{i,i+1}^{-1}(H_{2}^{T})_{ii}\\
 & =-2t\frac{1}{t^{i-1}}(i-1)t^{i-2}\\
 & =-2(i-1).
\end{align*}
\end{enumerate}
From these three cases we obtain that
\[
T_{23}=-2t(H_{2}^{T})^{-1}(H_{2}')^{T}+2t(H_{2}^{T})^{-1}QH_{2}^{T}=-2\mathrm{diag}(0,1,\ldots, n-1).
\]
As a result, we have
\begin{equation}
[(2n-1)I-2t(H_{2}^{T})^{-1}(H_{2}')^{T}+2t(H_{2}^{T})^{-1}QH_{2}^{T}]=\mathrm{diag}(2n-1,2n-3,\ldots,1)\label{eq34-1},
\end{equation}
which is \eqref{eq34}. This finishes the proof of the lemma.
\end{proof}
The following lemma is a purely combinatoric identity. We expect that this identity is not new, but we include a proof here for the sake of completeness.
\begin{lemma}
\label{lem: comb identity}
It holds that
\begin{eqnarray}
\sum_{j=0}^{k}\left(\begin{array}{c}
n\\
j
\end{array}\right)(-1)^{j} & =(-1)^{k} & \left(\begin{array}{c}
n-1\\
k
\end{array}\right)\label{eq3-1}
\end{eqnarray}
\end{lemma}
\begin{proof}
We prove \eqref{eq3-1} by induction. Obviously, (\ref{eq3-1}) is correct for $k=0$. Assume that we (\ref{eq3-1}) is valid with some value of $k$.

We will prove (\ref{eq3-1}) is also valid for $k+1$. We have
\begin{align*}
\sum_{j=0}^{k+1}\left(\begin{array}{c}
n\\
j
\end{array}\right)(-1)^{j} & =\sum_{j=0}^{k}\left(\begin{array}{c}
n\\
j
\end{array}\right)(-1)^{j}+\left(\begin{array}{c}
n\\
k+1
\end{array}\right)(-1)^{k+1}\\
 & =(-1)^{k}\left(\begin{array}{c}
n-1\\
k
\end{array}\right)+\left(\begin{array}{c}
n\\
k+1
\end{array}\right)(-1)^{k+1}\\
 & =(-1)^{k}\left(\left(\begin{array}{c}
n-1\\
k
\end{array}\right)-\left(\begin{array}{c}
n\\
k+1
\end{array}\right)\right)\\
 & =(-1)^{k+1}\left(\begin{array}{c}
n-1\\
k+1
\end{array}\right),
\end{align*}
where in the last equality we have used the fact that $\left(\begin{array}{c}
n-1\\
k
\end{array}\right)+\left(\begin{array}{c}
n-1\\
k+1
\end{array}\right)=\left(\begin{array}{c}
n\\
k+1
\end{array}\right)$. This is \eqref{eq3-1} for $k+1$. Therefore by the induction principle, (\ref{eq3-1}) is proved.
\end{proof}
\begin{lemma} For all $j\leq k $ and $k-j\le n$, we have
\begin{align}
\sum_{i=1}^{j}\frac{(-1)^{i+j}}{(i-1)!(j-i)!}\frac{(n+i-1)!}{(n+i-k)!} & =\begin{pmatrix}k-1\\
j-1
\end{pmatrix}\frac{n!}{(n-(k-j))!}.\label{eq2}
\end{align}
\end{lemma}
\begin{proof}
The identity \eqref{eq2} has been used and proved in our previous paper \cite[Proof of Theorem 4.1]{DuongTran16}. For the convenience of the reader we provide a proof here.
Consider the function $f(x)=x^{n}(1-x)^{j-1}$. On the one hand, we have 
\begin{align*}
f(x) & =x^{n}\sum_{i=0}^{j-1}\begin{pmatrix}j-1\\
i
\end{pmatrix}(-1)^{i}x^{i}\\
 & =\sum_{i=0}^{j-1}\begin{pmatrix}j-1\\
i
\end{pmatrix}(-1)^{i}x^{n+i}\\
 & =\sum_{i=1}^{j}\begin{pmatrix}j-1\\
i-1
\end{pmatrix}(-1)^{i-1}x^{n+i-1}.
\end{align*}
Therefore,
\begin{align*}
f(x)^{(k-1)} & =\sum_{i=1}^{j}\begin{pmatrix}j-1\\
i-1
\end{pmatrix}(-1)^{i-1}\frac{(n+i-1)!}{(n+i-1-(k-1))!}x^{n+i-1-(k-1)}\\
 & =\sum_{i=1}^{j}\frac{(j-1)!}{(j-i)!(i-1)!}(-1)^{i-1}\frac{(n+i-1)!}{(n+i-k)!}x^{n+i-1-(k-1)}.
\end{align*} 
It follows that
\begin{equation*}
\sum_{i=1}^{j}\frac{(-1)^{i}}{(i-1)!(j-i)!}\frac{(n+i-1)!}{(n+i-k)!}=\frac{-1}{(j-1)!}f(1)^{(k-1)},
\end{equation*}
and by multiplying both sides of this equality with $(-1)^j$, we get
\begin{align}
\label{eq: id1}
\sum_{i=1}^{j}\frac{(-1)^{i+j}}{(i-1)!(j-i)!}\frac{(n+i-1)!}{(n+i-k)!} & =\frac{(-1)^{j+1}}{(j-1)!}f(1)^{(k-1)}.
\end{align}
On the other hand, according to Leibniz formula, we have
\begin{align*}
f(x)^{(k-1)} & =[x^{n}(1-x)^{j-1}]^{(k-1)}\\
 & =\sum_{i=0}^{k-1}\begin{pmatrix}k-1\\
i
\end{pmatrix}[x^{n}]^{(k-1-i)}[(1-x)^{j-1}]^{(i)}.
\end{align*}
In particular, we get
\begin{align}
f(x)^{(k-1)}|_{x=1} & =\sum_{i=0}^{k-1}\begin{pmatrix}k-1\\
i
\end{pmatrix}[x^{n}]^{(k-1-i)}[(1-x)^{j-1}]^{(i)}|_{x=1}\nonumber\\
 & =\sum_{i=0}^{j-1}\begin{pmatrix}k-1\\
i
\end{pmatrix}[x^{n}]^{(k-1-i)}[(1-x)^{j-1}]^{(i)}|_{x=1}\nonumber\\
 & \qquad+\sum_{i=j}^{k-1}\begin{pmatrix}k-1\\
i
\end{pmatrix}[x^{n}]^{(k-1-i)}[(1-x)^{j-1}]^{(i)}|_{x=1}\nonumber\\
 & =\begin{pmatrix}k-1\\
j-1
\end{pmatrix}[x^{n}]^{(k-1-(j-1))}(-1)^{j-1}[(1-x)^{j-1}]^{(j-1)}|_{x=1}\nonumber\\
 & =\begin{pmatrix}k-1\\
j-1
\end{pmatrix}\frac{n!}{(n-(k-j)!)}x^{n-(k-j)}(j-1)!(-1)^{j-1}|_{x=1}\nonumber\\
 & =\begin{pmatrix}k-1\\
j-1
\end{pmatrix}\frac{n!}{(n-(k-j))!}(j-1)!(-1)^{j-1}.\label{eq: id2}
\end{align}
The identity \eqref{eq2} is then followed from \eqref{eq: id1}  and \eqref{eq: id2}.
\end{proof}
Having these lemmas, we are ready to prove Proposition \ref{prop: aux} in the following subsections.
\subsubsection{Proof of Proposition \ref{prop: aux}--(1): $T_{1}$ anti-symmetric}
In this section,we prove that $T_1$ is anti-symmetric where
\[
T_1=(2n-1)H_{1}^{T}MH_{1}-2t(H_{1}')^{T}MH_{1}-t^{2-2n}H_{1}^{T}MH_{2}DH_{2}^{T}MH_{1}.
\]
We have
\begin{align*}
 & t(H_{1}^{T})^{-1}(H_{1}')^{T}=\mathrm{diag}(0,1,...,n-1)\quad \text{and}\quad (2n-1)M=\mathrm{diag}(2n-1,\ldots,2n-1)M
\end{align*}
Therefore, we get
\begin{align*}
T_{1} &=(2n-1)H_{1}^{T}MH_{1}-2t(H_{1}')^{T}MH_{1}-t^{2-2n}H_{1}^{T}MH_{2}DH_{2}^{T}MH_{1}
\\&=H_{1}^{T}[(2n-1)M-2t(H_{1}^{T})^{-1}(H_{1}')^{T}M-t^{2-2n}MH_{2}DH_{2}^{T}M]H_{1}\\
& =H_{1}^{T}[\mathrm{diag}(2n-1,\ldots,3,1)M-t^{2-2n}MH_{2}DH_{2}^{T}M]H_{1}\\
 &\overset{(*)}{=}H_{1}^{T}[\mathrm{diag}(2n-1,\ldots,3,1)M-MH_{0}M]H_{1}\\
 & =H_{1}^{T}M[M^{-1}\mathrm{diag}(2n-1,\ldots,3,1)-H_{0}]MH_{1},
\end{align*}
where we have used Lemma \ref{lem: Ho} to obtain $(*)$. Let
\begin{align}
T_{11}:=M^{-1}\mathrm{diag}(2n-1,\ldots,3,1)-H_{0}=AB^{-1}\mathrm{diag}(2n-1,\ldots,3,1)-H_{0}. \label{eq1}
\end{align}
To prove $T_1$ is anti-symmetric, it suffices to prove that $T_{11}$ is anti-symmetric. Let $C_{ij}$ be the element $(i,j)$ of $T_{11}$. Then $C_{ij}$ can be computed as follows.
\begin{align}
C_{ij} & =(2(n-j)+1)\sum_{k}A_{ik}(B^{-1})_{kj}-(H_{0})_{ij}\nonumber\\
 & =(2(n-j)+1)\sum_{k+j\le n+1}A_{ik}(B^{-1})_{kj}-(H_{0})_{ij}\nonumber \\
 & =(2(n-j)+1)\sum_{\substack{1\le k\le n\\k+j\le n+1}}(i-1)!\begin{pmatrix}n+k-1\\
i-1
\end{pmatrix}(-1)^{k-1}\frac{1}{(n+k-1)!(n+1-k-j)!}-\frac{1}{(n-i)!(n-j)!}\nonumber \\
 & =-\frac{1}{(n-i)!(n-j)!}+(2(n-j)+1)\sum_{\substack{1\le k\le n\\k+j\le n+1}}(i-1)!\begin{pmatrix}n+k-1\\
i-1
\end{pmatrix}(-1)^{k-1}\frac{1}{(n+k-1)!(n+1-k-j)!}\nonumber \\
 & =-\frac{1}{(n-i)!(n-j)!}+(2(n-j)+1)\sum_{k=1}^{n+1-j}(i-1)!\frac{(n+k-1)!}{(i-1)!(n+k-i)!}(-1)^{k-1}\frac{1}{(n+k-1)!(n+1-k-j)!}\nonumber\\
 & =-\frac{1}{(n-i)!(n-j)!}+(2(n-j)+1)\sum_{k=1}^{n+1-j}\frac{(-1)^{k-1}}{(n+k-i)!(n+1-k-j)!}\nonumber\\
 & =-\frac{1}{(n-i)!(n-j)!}+(2(n-j)+1)\sum_{l=0}^{n-j}\frac{(-1)^{n-j-l}}{(2n+1-i-j-l)!l!},\label{eq:27}
\end{align}
where to obtain the last equality, we have changed variable $l:=n+1-k-j$.

Applying (\ref{eq3-1}) for $n=p+q+1$ ane $k=q$ we obtain
\begin{alignat}{2}
&&\sum_{l=0}^{q}\left(\begin{array}{c}
p+q+1\\
l
\end{array}\right)(-1)^{l} & =(-1)^{q}\left(\begin{array}{c}
p+q\\
q
\end{array}\right)\nonumber\\
&\iff & \sum_{l=0}^{q}\frac{(p+q+1)!}{(p+q+1-l)!l!}(-1)^{l} & =(-1)^{q}\frac{(p+q)!}{p!q!}\nonumber \\
&\iff &\sum_{l=0}^{q}\frac{1}{(p+q+1-l)!l!}(-1)^{l} & =(-1)^{q}\frac{1}{(p+q+1)p!q!}\nonumber \\
& \iff & -\frac{1}{p!q!}+(2q+1)(-1)^{q}\sum_{l=0}^{q}\frac{1}{(p+q+1-l)!l!}(-1)^{l} & =-\frac{1}{p!q!}+\frac{(2q+1)}{(p+q+1)p!q!}\label{eq26-1}.
\end{alignat}
Now applying (\ref{eq26-1}) for $p=n-i$ and $q=n-j$ we get
\begin{align}
&-\frac{1}{(n-i)!(n-j)!}+(2(n-j)+1)\sum_{l=0}^{n-j}\frac{(-1)^{n-j-l}}{(2n+1-i-j-l)!l!}\nonumber
\\&\qquad =-\frac{1}{(n-i)!(n-j)!}+\frac{(2(n-j)+1)}{(2n+1-i-j)(n-i)!(n-j)!}\nonumber
\\&\qquad=\frac{i-j}{(2n+1-i-j)(n-i)!(n-j)!}.
\label{eq26-2}
\end{align}

From \eqref{eq26-2} and \eqref{eq:27} we achieve
\[
C_{ij} =\frac{i-j}{(2n+1-i-j)(n-i)!(n-j)!},
\]
which implies that $C_{ji} =-C_{ji}$. Therefore, $T_{11}$ is anti-symmetric and so is $T_1$. 
\subsubsection{Proof of Proposition \ref{prop: aux}--(2): $T_{2}=0$ }
In this section, we prove that $T_2=0$ where
\[
T_{2}=(1-2n)H_{2}^{T}MH_{1}+t\big((H_{2}')^{T}MH_{1}+H_{2}^{T}\,M\,H_{1}'\big)-tQH_{2}^{T}MH_{1}+H_{2}^{T}MH_{0}MH_{1}.
\]
According to Lemma \ref{lem: diag relation}, we have
\[
-2t(H_{2}^{T})^{-1}(H_{2}')^{T}+2t(H_{2}^{T})^{-1}QH_{2}^{T}=-2\mathrm{diag}(0,1,\ldots,n-1),
\]
implying that
\[
-t(H_{2}')^{T}+tQH_{2}^{T}=-H_{2}^{T}\mathrm{diag}(0,...,n-1).
\]
Therefore,
\[
T_{2}=(1-2n)H_{2}^{T}MH_{1}+H_{2}^{T}\mathrm{diag}(0,...,n-1)MH_{1}+tH_{2}^{T}MH_{1}'+H_{2}^{T}MH_{0}MH_{1}.
\]
Since $tH_{1}'H_{1}^{-1}=\mathrm{diag}(0,...,n-1)$, we have 
\[
tH_{2}^{T}MH_{1}'=tH_{2}^{T}MH_{1}'H_{1}^{-1}H_{1}=H_{2}^{T}M\mathrm{diag}(0,...,n-1)H_{1}.
\]
Therefore, we get
\begin{align*}
T_{2} & =(1-2n)H_{2}^{T}MH_{1}+H_{2}^{T}\mathrm{diag}(0,...,n-1)MH_{1}+H_{2}^{T}M\mathrm{diag}(0,..,n-1)H_{1}+H_{2}^{T}MH_{0}MH_{1}\\
 & =H_{2}^{T}[(1-2n)I+\mathrm{diag}(0,\ldots,n-1)+M\mathrm{diag}(0,...,n-1)M^{-1}+MH_{0}]MH_{1}\\
 & =H_{2}^{T}M[M^{-1}\mathrm{diag}(1-2n,\ldots,-n)+\mathrm{diag}(0,...,n-1)M^{-1}+H_{0}]MH_{1}.
\end{align*}

We will prove that $T_{22}=0$ where
\begin{align*}
T_{22}&=M^{-1}\mathrm{diag}(1-2n,\ldots,-n)+\mathrm{diag}(0,...,n-1)M^{-1}+H_{0}
\\&=AB^{-1}\mathrm{diag}(1-2n,...,-n)+\mathrm{diag}(0,...,n-1)AB^{-1}+H_{0}.
\end{align*}
Note that the last equality is because $M^{-1}=AB^{-1}$. We compute each element of $T_{22}$ as follows.
\begin{align*}
(T_{22})_{ij} & =\sum_{k}a_{ik}(B^{-1})_{kj}(j-2n)+(i-1)\sum_{k}A_{ik}(B^{-1})_{kj}+(H_{0})_{ij}\\
 & =\sum_{k}a_{ik}(B^{-1})_{kj}(j-2n+i-1)+(H_{0})_{ij}\\
 & =(j-2n+i-1)\left[-\frac{1}{(2n-j)(n-i)!(n-j)!}+\frac{(2(n-j)+1)}{(2n-j)(2n+1-i-j)(n-i)!(n-j)!}+\frac{(H_{0})_{ij}}{(2n-j)}\right]+(H_{0})_{ij}\\
 & =\frac{(j-2n+i-1)}{(2n-j)}\left[-\frac{1}{(n-i)!(n-j)!}+\frac{(2(n-j)+1)}{(2n+1-i-j)(n-i)!(n-j)!}\right]+\frac{(j-2n+i-1+2n-j)}{2n-j}(H_{0})_{ij}\\
 & =\frac{(j-2n+i-1)}{(2n-j)(n-i)!(n-j)!}\left[-1+\frac{(2(n-j)+1)}{(2n+1-i-j)}\right]+\frac{(i-1)}{2n-j}\frac{1}{(n-i)!(n-j)!}\\
 & =\frac{1}{(2n-j)(n-i)!(n-j)!}\left[-(j-2n+i-1)-(2(n-j)+1)+i-1\right]\\
 & =0
\end{align*}
Therefore $T_{22}=0$. It follows that $T_{2}=0$.

\subsubsection{Proof of Proposition \ref{prop: aux}--(3): $T_{3}$ anti-symmetric:}
In this section, we prove that $T_3$ is anti-symmetric where
\[
T_3=(2n-1)H_{2}^{T}MH_{2}-2t(H_{2}')^{T}MH_{2}+2tQH_{2}^{T}MH_{2}-t^{2-2n}H_{2}^{T}MH_{2}DH_{2}^{T}MH_{2}.
\]
According to Lemma \ref{lem: Ho} $H_{2}DH_{2}^{T}=H_0$; therefore we have
\begin{align*}
T_{3} & =H_{2}^{T}[(2n-1)M-2t(H_{2}^{T})^{-1}(H_{2}')^{T}M+2t(H_{2}^{T})^{-1}QH_{2}^{T}M-MH_{0}M]H_{2}\\
 & =H_{2}^{T}M[M^{-1}(2n-1)-M^{-1}2t(H_{2}^{T})^{-1}(H_{2}')^{T}+M^{-1}2t(H_{2}^{T})^{-1}QH_{2}^{T}-H_{0}]MH_{2}\\
 & =H_{2}^{T}M\left(AB^{-1}[(2n-1)I-2t(H_{2}^{T})^{-1}(H_{2}')^{T}+2t(H_{2}^{T})^{-1}QH_{2}^{T}]-H_{0}\right)MH_{2}
\end{align*}

We define 
\[
T_{31}:=AB^{-1}[(2n-1)I-2t(H_{2}^{T})^{-1}(H_{2}')^{T}+2t(H_{2}^{T})^{-1}QH_{2}^{T}]-H_{0}.
\]
To prove $T_3$ is anti-symmetric, it suffices to show that $T_{31}$ is anti-symmetric.

From \eqref{eq:27} we have:
\[
(2(n-j)+1)\sum_{k}A_{ik}(B^{-1})_{kj}-(H_{0})_{ij} =\frac{i-j}{(2n+1-i-j)(n-i)!(n-j)!},
\]
which implies that (using the fact that $(H_{0})_{ij}=\frac{1}{(n-i)!(n-j!)}$)
\begin{equation}
(AB^{-1})_{ij}=\sum_{k}A_{ik}(B^{-1})_{kj}=\frac{1}{(2n+1-i-j)(n-i)!(n-j)!}\label{eq34-2}.
\end{equation}
Using \eqref{eq34-2} and Lemma \ref{lem: diag relation}, we obtain
\begin{align*}
(T_{31})_{ij} & =\frac{1}{(2n+1-i-j)(n-i)!(n-j)!}(2(n-j)+1)-\frac{1}{(n-i)!(n-j)!}\\
 & =\frac{i-j}{(2n+1-i-j)(n-i)!(n-j)!}.
\end{align*}
Therefore, $T_{31}$ is anti-symmetric and so is $T_3$.
\subsubsection{Proof of Proposition \ref{prop: aux}--(4): $\mathrm{Tr}(DH_{2}^{T}MH_{2})=n^2\,d\,t^{2(n-1)}$}
In this section, we prove that last identity in Proposition \ref{prop: aux}. We will show that
\[
\mathrm{Tr}(DH_{2}^{T}MH_{2})=n^2\,d\,t^{2(n-1)}.
\]

We recall that all the matrices $A, B,H_2, H_1, D, M, L$ and $U$ are of order $dn$. Each entry of these matrix should be understood as a matrix of order $d$ that equals to the entry multiply with the $d$-dimensional identity matrix $I_d$.

Since $D=\mathrm{diag}(0,\ldots,0,1)$, we have $DH_{2}^{T}MH_{2}=\mathrm{diag}(0,\ldots,0,[(H_{2})^{T}MH_{2}]_{nn})$.
Therefore 
\begin{align*}
\mathrm{Tr}(DH_{2}^{T}MH_{2})&=\mathrm{Tr} [(H_{2})^{T}MH_{2}]_{nn}\\
 & =\mathrm{Tr}\Big[\sum_{k,l}(H_{2})_{nl}M_{lk}(H_{2}^{T})_{kn}\Big]\\
 & =\mathrm{Tr}\Big[\sum_{k,l}(H_{2})_{nl}M_{lk}(H_{2})_{nk}\Big]\\
 & =t^{2(n-1)}\mathrm{Tr}[M_{nn}].
\end{align*}
Next we show that $M_{nn}=n^{2} I_d$. According to Theorem \ref{theo: LU of A 2}, we have
\[
M=BA^{-1}=BU^{-1}L^{-1},
\]
where
\[
B_{nk}=\frac{(n+k-1)!}{(k-1)!},\quad(U^{-1})_{kl}=\begin{cases}
\frac{(-1)^{k+l}}{(k-1)!(l-k)!} & l\geq k\\
0 & l<k.
\end{cases},\quad\text{and}\quad (L^{-1})_{ln}=\begin{cases}
1 & l=n\\
0 & l\neq n
\end{cases}.
\]
Therefore, 
\begin{align*}
M_{nn} & =\sum_{k,l}B_{nk}(U^{-1})_{kl}(L^{-1})_{ln}\\
 & =\sum_{k}B_{nk}(U^{-1})_{kn}\\
 & =\sum_{k}\frac{(n+k-1)!}{(k-1)!}\frac{(-1)^{k+n}}{(k-1)!(n-k)!}
\end{align*}
Now we show that
\[
n^{2}=\sum_{i=1}^{n}\frac{(n+i-1)!}{(i-1)!}\frac{(-1)^{i+n}}{(i-1)!(n-i)!}.
\]
Applying \eqref{eq2} for $j=n$ and $k=n+1$ we get
\[
\sum_{i=1}^{n}\frac{(-1)^{i+n}}{(i-1)!(n-i)!}\frac{(n+i-1)!}{(n+i-(n+1))!}=\begin{pmatrix}n+1-1\\
n-1
\end{pmatrix}\frac{n!}{(n-(n+1-n))!},
\]
which implies that
\[
\sum_{i=1}^{n}\frac{(-1)^{i+n}}{(i-1)!(n-i)!}\frac{(n+i-1)!}{(i-1)!}=n^{2}
\]
as desired, that is $M_{nn}=n^2I_d$.

Therefore we get $\mathrm{Tr}(DH_{2}^{T}MH_{2})=t^{2(n-1)}\rm{Tr}(M_{nn})=t^{2(n-1)}n^2 d$ as stated.
\subsection{Proof of Lemma \ref{lem: InverH2H1}}
In this secion we prove Lemma \ref{lem: InverH2H1}: we need to show that $H_{2}^{-1}H_{1}=H$ where $H_1, H_2$ and $H$ are given in  \eqref{eq: H1 and H2} and \eqref{eq: H}.

We will show that $(H_{2}H)_{ik}=(H_{1})_{ik}$
for all $i,k$. 
\begin{itemize}
\item If $k<i$, we have $(H_{2})_{ij}=0$ if $j\le k$ and $H_{jk}=0$
if $j>k$. Therefore,
\begin{equation*}
(H_2H)_{ik}=\sum_{j=1}^{n}(H_{2})_{ij}H_{jk}=\sum_{j=1}^{k}(H_{2})_{ij}H_{jk}+\sum_{j=k+1}^{n}(H_{2})_{ij}H_{jk}=0=(H_{1})_{ik}.
\end{equation*}
\item If $k>i$, we have $(H_{2})_{ij}=0$ if $j\le i$ and $H_{jk}=0$
if $j>k$. Therefore,
\begin{align*}
(H_2H)_{ik}=\sum_{j=1}^{n}(H_{2})_{ij}H_{jk} & =\sum_{j=i}^{k}\frac{h^{j}}{(j-i)!}\frac{(-1)^{k-j}h^{k-j}}{(k-j)!}\\
 & =\sum_{j=i}^{k}\frac{(-1)^{k-j}h^{k-j}}{(j-i)!(k-j)!}\\
 & =\frac{1}{(k-i)!}\sum_{j=i}^{k}\frac{(k-i)!(-1)^{k-j}}{(j-i)!(k-j)!}\\
 & =\frac{1}{(k-i)!}(1+(-1))^{k-i}\\
 & =0=(H_{1})_{ik}.
\end{align*}

\item If $k=i,$ we have $(H_{2})_{ij}=0$ if $j<i$ and $H_{ji}=0$ if
$j>k$. Therefore,
\begin{align*}
(H_2H)_{ik}=\sum_{j=1}^{n}(H_{2})_{ij}H_{ji}=(H_{2})_{ii}H_{ii}=h^{i-1}=(H_{1})_{ik}.
\end{align*}
\end{itemize}
It follows from these cases that $H_{2}H=H_{1}$.

\subsection{Proof of Lemma \ref{lem:K}} 
In this section, we prove Lemma \ref{lem:K}. First we will show that $(H_{2}^{T})^{-1}=P$ where 
\[
P_{lj}=\begin{cases}
0 & \text{ if }l<j\\
\frac{(-1)^{l-j}}{(l-j)!h^{j-1}} & \text{ if }l\ge j.
\end{cases}
\]
We consider the following cases:
\begin{itemize}
\item If $i=j,$ we have $P_{lj}=0$ if $l<i$ and $(H_{2}^{T})_{il}=0$
if $l>i$. Therefore, 
\begin{align*}
\sum_{l=1}^{n}(H_{2}(h))_{il}P{}_{lj} & =(H_{2})_{ii}P{}_{ii}\\
 & =h^{i-1}\frac{1}{h^{i-1}}\\
 & =1.
\end{align*}
\item If $i<j$, we have $P_{lj}=0$ if $l\le i$ and $(H_{2}^{T}(h))_{il}=0$
if $l>i$. Therefore, 
\begin{align*}
\sum_{l=1}^{n}(H_{2})_{il}(H_{2}^{T})_{lj}^{-1} & =\sum_{l=1}^{i}(H_{2})_{il}P_{lj}+\sum_{l=i+1}^{n}(H_{2})_{il}P_{lj}\\
 & =0.
\end{align*}

\item If $i>j$, we have $P_{lj}=0$ if $l<j$ and $(H_{2})_{il}=0$
if $l>i$. Therefore, 
\begin{align*}
\sum_{l=1}^{n}(H_{2})_{il}P_{lj} & =\sum_{l=j}^{i}(H_{2})_{il}P_{lj}\\
 & =\sum_{l=j}^{i}\frac{h^{i-1}}{(i-l)!}\frac{(-1)^{l-j}}{(l-j)!h^{j-1}}\\
 & =\frac{h^{i-j}}{(i-j)!}\sum_{l=j}^{i}\frac{(i-j)!}{(i-l)!(l-j)!}\\
 & =\frac{h^{i-j}}{(i-j)!}(1-1)^{i-j}\\
 & =0.
\end{align*}
\end{itemize}
Therefore, $(H_{2}^{T})^{-1}=P$. We have $\K=(H_2^T M H_1)^{-1}=H_1^{-1} M^{-1} (H_2^T)^{-1}$. Therefore,
\begin{align}
\K_{ij}= & (H_{1}^{-1}M^{-1}(H_{2}^{T})^{-1})_{ij}\nonumber \\
= & \sum_{kl}(H_{1}^{-1})_{ik}(M^{-1})_{kl}(H_{2}^{T})_{lj}^{-1}\text{ (Note that }(H_{1}^{-1})_{ik}=0\text{ if }i\ne k)\nonumber \\
= & \sum_{k=i,l=1}^{n}(H_{1}^{-1})_{ik}(M^{-1})_{kl}(H_{2}^{T})_{lj}^{-1}\nonumber \\
= & \sum_{l=1}^{n}(H_{1}^{-1})_{ii}(M^{-1})_{il}(H_{2}^{T})_{lj}^{-1}\nonumber \\
= & \sum_{l=1}^{n}t^{1-i}\frac{1}{(2n+1-i-l)(n-i)!(n-l)!}(H_{2}^{T})_{lj}^{-1}\nonumber \\
 & \text{ (Note that }(M^{-1})_{ij}=\frac{1}{(2n+1-i-j)(n-i)!(n-j)!},\text{ according to }\eqref{eq34-2})\nonumber \\
= & \sum_{l=j}^{n}t^{1-i}\frac{1}{(2n+1-i-l)(n-i)!(n-l)!}\frac{(-1)^{l-j}}{(l-j)!t^{j-1}}\nonumber \\
= & \frac{h^{1-i}}{h^{j-1}}\sum_{l=j}^{n}\frac{1}{(2n+1-i-l)(n-i)!(n-l)!}\frac{(-1)^{l-j}}{(l-j)!}.\label{eqk}
\end{align}
We now transform further the last expression. Let $\beta=n-j,\alpha=n-i,u=n-j-l$, and $f(x)=\sum_{u=0}^{\beta}\frac{1}{(\alpha+1+u)}\frac{\beta!(-1)^{\beta-u}}{u!(\beta-u)!}x^{u+\alpha}$. From \eqref{eqk}, $\K_{ij}$ can be rewritten as follows
\begin{align}
\K_{ij} & =h^{-i-j+2}\sum_{l=0}^{n-j}\frac{1}{(2n+1-i-j-l)(n-i)!(n-j-l)!}\frac{(-1)^{l}}{l!}\nonumber \\
 & =\frac{h^{-i-j+2}}{\beta!\alpha!}\sum_{u=0}^{\beta}\frac{1}{(\alpha+1+u)}\frac{\beta!(-1)^{\beta-u}}{u!(\beta-u)!},\\
 & =\frac{h^{-i-j+2}}{\beta!\alpha!}f(1).\label{eqk2}
\end{align}
We now compute $f(1)$ in an alternative way. We have 
\begin{align*}
f(x) & =\int_{0}^{x}\sum_{u=0}^{\beta}\frac{\beta!(-1)^{\beta-u}}{u!(\beta-u)!}t^{u+\alpha}dt\\
 & =\int_{0}^{x}t^{\alpha}(t-1)^{\beta}dt.
\end{align*}
In particular, taking $x=1$ yields
\begin{equation}
f(1)=\int_{0}^{1}t^{\alpha}(t-1)^{\beta}dt.\label{eqk3}
\end{equation}

We now compute the integral on the RHS of \eqref{eqk3}. We have
\begin{align}
\int_{0}^{1}t^{\alpha}(t-1)^{\beta}dt & =\frac{1}{\alpha+1}\int_{0}^{1}(t-1)^{\beta}dt^{\alpha+1}\nonumber \\
 & =\frac{1}{\alpha+1}(t-1)^{\beta}t^{\alpha+1}|_{t=0}^{1}-\frac{\beta}{\alpha+1}\int_{0}^{1}t^{\alpha+1}(t-1)^{\beta-1}dt\nonumber \\
 & =-\frac{\beta}{\alpha+1}\left(\frac{1}{\alpha+2}(t-1)^{\beta-1}t^{\alpha+2}|_{t=0}^{1}-\frac{\beta-1}{\alpha+2}\int_{0}^{1}t^{\alpha+1}(t-1)^{\beta-1}dt\right)\nonumber \\
 & =\frac{\beta}{\alpha+1}\frac{\beta-1}{\alpha+2}\int_{0}^{1}(t-1)^{\beta-2}t^{\alpha+2}dt\nonumber \\
 & =...\nonumber \\
 & =(-1)^{\beta}\frac{\beta}{\alpha+1}\frac{\beta-1}{\alpha+2}...\frac{1}{\alpha+\beta}\int_{0}^{1}t^{\alpha+\beta}dt\nonumber \\
 & =(-1)^{\beta}\frac{\beta!}{(\alpha+1)...(\alpha+\beta)}\frac{1}{\alpha+\beta+1}t^{\alpha+\beta+1}|_{t=0}^{1}\nonumber \\
 & =(-1)^{\beta}\frac{\beta!\alpha!}{(\alpha+\beta+1)!}.\label{eqk4}
\end{align}

It follows from \eqref{eqk2},\eqref{eqk3} and \eqref{eqk4}) that
\begin{align*}
\K_{ij} & =\frac{h^{-i-j+2}}{\beta!\alpha!}(-1)^{\beta}\frac{\beta!\alpha!}{(\alpha+\beta+1)!}\\
 & =\frac{1}{h^{j+i-2}}\frac{(-1)^{n-j}}{(2n-i-j+1)!},
\end{align*}
which finishes the proof of Lemma \ref{lem:K}.
\section*{Acknowledgements}
The majority of this paper was written when M. H. Duong was at the Mathematics Institute, University of Warwick and was supported by ERC Starting Grant 335120 while H. M. Tran was at the Department of Industrial Engineering, Texas A\&M University.

\bibliographystyle{alpha}
\bibliography{bibDUONG}

\end{document}